\documentclass{llncs}
\usepackage{amsmath,amsfonts}
\usepackage{url}
\usepackage{algorithmicx}
\usepackage[ruled]{algorithm}
\usepackage{algpseudocode}
\usepackage{algpascal}
\usepackage{algc}
\usepackage{algorithm}
\usepackage{xcolor}
\newcommand{\w}{\mathcal{W}}

\newcommand{\cO}{\mathcal{O}}
\newcommand{\cT}{\mathcal{T}}
\newcommand{\cP}{\mathcal{P}}

\newcommand{\E}{\mathbf{E}}
\newcommand{\M}{\mathbf{M}}
\newcommand{\J}{\mathbf{J}}

\newcommand{\TE}{\mathbf{TE}}

\newcommand{\s}{\mathbf{S}}
\newcommand{\D}{\mathbf{D}}

\newcommand{\ch}{\operatorname{char}\nolimits}

\newcommand{\F}{\mathbb{F}}

\newcommand{\Z}{\mathbb{Z}}

\newcommand{\W}{\mathbf{W}}

\def\F{\mathbb{F}}

\newcommand{\keywords}[1]{{
    \list{}{\advance\topsep by -5ex \relax\small \leftmargin=1cm
      \labelwidth=0pt \listparindent=0pt \itemindent\listparindent
      \rightmargin\leftmargin}\item[\hskip\labelsep \bfseries
    Keywords:] {#1} \endlist}}

\institute{}
\author{ Seyed Gholamhossein Hosseini, Reza Rezaeian Farashahi\inst{1,2} } \institute{
	Department of Mathematical Sciences, Isfahan University of Technology,\\
	Isfahan, 84156-83111, Iran
	\\{\tt  g.hoseinimath@gmail.com,\; farashahi@cc.iut.ac.ir}\\
	\and
	School of Mathematics,
	Institute for Research in Fundamental Sciences (IPM),
	P.O. Box 19395-5746, Tehran, Iran\\
}

\begin{document}

\title{Differential addition on Jacobi quartic Curves}

\maketitle

\begin{abstract}
This paper presents new differential addition and doubling formulas for Jacobi quartic curves. Several  differential addition and doubling formulas are presented with  costs of $5\M+4\s+1\D$ , $3\M+7\s+1\D$ and $3\M+6\s+3\D$ 
when the given difference point is in affine form. Here, $\M, \s, \D$ denote the costs of a field
multiplication, a field squaring and a field multiplication by a constant, respectively. 
\end{abstract}



\section{Introduction}
An elliptic curve $E$ over a field $\mathbb{F}$ is traditionally defined by the so called Weiersras\ss\, equation
\[ y^2+a_1xy+a_3y=x^3+a_2x^2+a_4x+a_6\enspace \]
where coefficients $a_1, a_2, a_3, a_4$ and $a_6$ are in $\mathbb{F}$.
The cryptographic application of elliptic curve over finite fields was proposed by Koblitz \cite{K87} and Miler \cite{M86}. 
Since the introduction of elliptic curve cryptography (ECC), elliptic curves over finite fields have been studied intensively and in particular, many proposals have been made to speed up their group arithmetic \cite{HBook,WDC8}. ECC is one of the attractive asymmetric key cryptosystems with the main advantage of achieving smaller key sizes under the same security level compare to that of other existing asymmetric systems such as RSA. This makes ECC suitable for software and hardware implementation in constrained environments including RFID tags, mobiles, sensors, and smart cards. 
 In particular, the group arithmetic of many known forms of elliptic curves such as cubic Weierstrass, Hessian, quartic equations and intersection of two quadratic surfaces~\cite{HBook,CH86} are revisited. In addition, in the last decade variant forms of Edwards curves \cite{BBLP,BerLan1,Edw,HWCD8} are presented. These families of cubic and quartic elliptic curves are extended to cover more isomorphism classes over finite fields, and to provide more flexibility for faster and efficient group arithmetic. \cite{AG,Far1,FMW,FS}

The use of Jacobi curves in cryptology is explained by Chudnovsky and Chudnovsky \cite{CH86}, Liardet and Smart \cite{LS01}, Billet and Joye \cite{BJ}.
Jacobi curves provide resistance to simple side channel attacks due to their efficient unified point addition formulas. Duquesne \cite{D7},  
Hisil et.al. \cite{HWCD9,HWD9}, Pl\^{u}t \cite{P11}, Feng et.al. \cite{FNW} proposed several fast and efficient unified addition and doubling formulas on Jacobi curves and their variants. The complete addition and doubling formulas are given for Jacobi curves with properly chosen parameters. The computation of pairings on Jacobi quartic elliptic curves is also investigated in \cite{DF13,DMF14,WWZL11}. 

The main computational operation in the arithmetic of elliptic curve cryptography is the scalar multiplication, that is to compute $kP$ for a given point $P$ on elliptic curve $E$ defined over a finite field $\F_q$ and a given integer $k$. This operation is performed recursively by point addition and point doubling operations. One trivial important factor for  fast implementation of these basic curve operations is to reduce the number of required finite field operations. On the other hand, the time or power differences between implementing point addition and point doubling can reveal information about the bits of secret $k$ which makes the system insecure against side channel attacks. To obtain a fast and secure implementation of the scalar multiplication different forms of elliptic curves with several coordinates systems, and different scalar multiplication algorithms have been studied \cite{HBook}. 

The Montgomery ladder \cite{Mo} is a well known technique to perform scalar multiplication of points for a special so-called Montgomery curves. In each step of the Montgomery ladder both point addition and point doubling operations are performed which makes the cost of ground finite field operations almost uniform for every bit of the scalar as a secret. This makes the method to be resistance against side channel attacks. In more details, the basic formulas in each step of the Montgomery ladder is differential addition
and doubling expressed only by the $x$-coordinates of the points. That is, for given two points and their difference on elliptic curve $E$, to compute their addition and doubling.  
In other words, the $x$-coordinates of the points $2Q$ and $2Q+P$ are computed from the $x$-coordinates of the points $Q$, $Q+P$ and the fixed point $P$. This recursive method computes the $x$-coordinate of the point $kP$. Note that to avoid the costly field inversion operation, the computations are performed where points are represented in projective coordinates.     

The Montgomery ladder processes the scalar multiplication on any elliptic curve $E$ defined over $\F$. The basic computation in each step of the Montgomery ladder, is differential addition and doubling on $E$.  That is, instead of working with the full representations of the points, differential addition and doubling are computed on the images of points at some suitable rational function $w$ on the elliptic curve $E$. Several differential addition formulas for  different forms of elliptic curves in odd and even characteristic are presented, see e.g \cite{BLF8,BERNLANG,CGF,FH,FH2,GL9,GGX12}. 
Gaudry and Lubicz \cite{GL9} give a fast differential addition Montgomery-like formulas for Kummer line. 
A new fast formulas for complete twisted Edwards curves are presented by Farashahi and Hosseini \cite{FH2}. Also, the fast Montgomery-like formulas for twisted Edwards and Montgomery curves with full 2-torsion points are presented in \cite{FH2,BERNLANG}.

The family of Jacobi quartic curves over $\F_q$ is the family of all elliptic curves over $\F_q$ with a point of order 2. This family properly includes (up to $\F_q$-isomorphism) all the families of Legendre, Edwards, twisted Edwards and Montgomery curves. 
Notice, the family of Montgomery and twisted Edwards curves do not cover all elliptic curves with full 2-torsion subgroup.
In particular, the differential addition Montgomery-like formulas are presented for Jacobi curves either with a point of order 4 or with full 2-torsion points. This means, the differential addition Montgomery-like formulas are presented for all elliptic curves over $\F_q$ with a subgroup of order 4. In addition, the fast Montgomery-like formulas are provided for complete Jacobi curves if the parameters are chosen properly. 

The paper is organized as follows. A review of extended Jacobi quartic curves  is given in \S \ref{sec:JQ}. Differential addition on Montgomery curves  is introduced in \S \ref{sec-dadd}. A  review of isogeny is introduced in  \S \ref{Iso}. 
New fast and efficient differential addition and doubling formulas for Jacobi quartics curves  are introduced in \S \ref{sec-daddJ}.  Finally, \S \ref{sec-Compar} concludes with previous works. 

Throughout the paper, the letter $p$ always denotes an odd prime number and $q$ denotes a prime power of $p$. 
A field is denoted by $\F$ and a finite field of size $q$ is denoted by $\F_q$. 
Let $\chi$ denote the quadratic character in $\F_q$, where $p\ge 3$. Then, for any $q$ where $p\ge 3$, we have
$a=r^2$ for some $r \in \F^*_q$ if and only if $\chi(a) =1$. 


\section{Jacobi quartics } \label{sec:JQ}

Let $\F$ be a field with characteristic not equal 2 and let $\epsilon,\delta \in \F$ such that $\epsilon(\delta^2-\epsilon)\ne 0$. 
An {\it (extended) Jacobi quartic curve} over the field 
$\F$, is given by the equation
\begin{equation}
	\label{eq:EJacob} \J_{\epsilon,\delta}: ~ y^2 = \epsilon x^4+2\delta x^2+1.
\end{equation}
The Jacobi quartic curve $\J_{\epsilon,\delta}$ is isomorphic to an elliptic curve of Weierstrass form
$$\E_{\epsilon,\delta}: ~y^2=( x^2-4\epsilon)( x+2\delta )$$ 
under the following birational maps (see \cite[Theorem 2.17]{WDC8}):   

\begin{equation*}
\begin{array}{cc}
\zeta: \J_{\epsilon,\delta} \longrightarrow E_{\epsilon,\delta} & \zeta^{-1}: E_{\epsilon,\delta} \longrightarrow  \J_{\epsilon,\delta} \vspace{3pt}\\
\zeta (x,y)=\left(\frac{2y}{x^2-4\epsilon},\frac{x^2+4\delta x+4\epsilon}{x^2-4\epsilon}\right)~\quad & ,\quad ~\zeta^{-1}(x,y)=\left(\frac{2(y+1)}{x^2},\frac{4(y+1)+4\delta x^2}{x^3}\right).
\end{array}
\end{equation*}
%

The addition and doubling law for $\J_{\epsilon,\delta}$ are given by 
\begin{equation}
\begin{array}{c}
\label{eq:AdEJ}
(x_1,y_1)~,~(x_2,y_2)  \mapsto \vspace*{5pt}\\

\left(\dfrac{x_1y_2+x_2y_1}  {1-\epsilon x_1^2x_2^2},\, \dfrac{(y_1y_2+2\delta x_1x_2)(1+\epsilon x_1^2x_2^2)+2\epsilon x_1x_2(x_1^2+x_2^2)}  {(1-\epsilon x_1^2x_2^2)^2}\right),
\end{array}
\end{equation}

\begin{equation}
\label{eq:DbEJ}
\begin{array}{c}
2(x_1,y_1)  \mapsto \left(\dfrac{2x_1y_1}  {1-\epsilon x_1^4},\,\dfrac{(y_1^2+2\delta x_1^2)(1+\epsilon x_1^4)+4\epsilon x_1^4}  {(1-\epsilon x_1^4)^2} \right).
\end{array}
\end{equation}
The negative of a point $(x,y)$ is $(-x,y)$. This curve has two points $(0,1)$ and $(0,-1)$ which are invariant under the negation law. The point  $\cO=(0,1)$ is the identity element of addition law while  $(0,-1)$ is a point of order 2. Billet and Joye in \cite{BJ} remark that any elliptic curve with a point of order 2 can be expressed by a Jacobi quartic curve given by the Eq. \eqref{eq:EJacob}. 

This family strictly covers (up to $\F_q$-isomorphism) the families of Legendre, Edwards, twisted Edwards and Montgomery curves. Notice, there are Montgomery or twisted Edwards curves that can not be represented over $\F_q$ in the Jacobi form. 

We note that, the doubling formulas \eqref{eq:DbEJ} are obtained from the addition formulas \eqref{eq:AdEJ}.  
Hisil et.al. \cite{HWD9} show that if $\epsilon$ is a non-square in $\F$ then the addition formulas are complete, i.e., the addition formulas \eqref{eq:AdEJ} work for any input points in $\J_{\epsilon,\delta}(\F)$.
The projective closure of the Jacobi quartic curve $ \J_{\epsilon,\delta}$  in $\mathbb{P}^2$ includes the projective points $(X:Y:Z)$ in $\mathbb{P}^2(\F)$ satisfying the curve equation   
\begin{equation} \label{PJacob}
 Y^2Z^2 = \epsilon X^4+2\delta X^2Z^2+Z^4.
\end{equation}

 Hisil et.al.~\cite{HWD9} consider the projective closure of  $ \J_{\epsilon,\delta}$ embedded in $\mathbb{P}^3$ given by the projective homogeneous equations:  
\begin{equation*}
\begin{cases}
X^2-TZ=0 \quad ~\\
\quad ~\\
Y^2-\epsilon T^2-2\delta X^2-Z^2=0 
\end{cases}
\end{equation*}
%
In this extended projective homogeneous coordinates \cite{HWD9}, an affine point $(X,Y)$ represented by $(X:Y:T:Z)$ where $T=X^2/Z$. 
Here, the addition formulas are given by 
\begin{gather} \label{Ad:TE1}
(X_3 : Y_3 : T_3: Z_3)=(X_1 : Y_1 : T_1: Z_1)+(X_2 : Y_2 : T_2 : Z_2)  \hspace*{37pt} \nonumber\\
\begin{array}{c} \label{eq:TEH1}
X_3=(X_1Y_2+X_2Y_1)(Z_1Z_2-\epsilon T_1T_2), \\
Y_3= (Y_1Y_2+2\delta X_1X_2)(Z_1Z_2+\epsilon T_1T_2)+2\epsilon X_1X_2(T_1Z_2+T_2Z_1), \\
T_3= (X_1Y_2+X_2Y_1)^2,\\
Z_3=  (Z_1Z_2-\epsilon T_1T_2)^2.
\end{array}
\end{gather}
A point representation $(X:Y:Z)$ satisfying \eqref{PJacob} can be converted to the new coordinates by computing $(XZ:YZ:X^2:Z^2)$ with $Z \ne 0$. The identity point is $\cO=(0:1:0:1)$ and the additive negation of a point  $(X:Y:T:Z)$ is $(-X:Y:T:Z)$. 
The point $\mathcal{T}=(0:-1:0:1)$ is a point of order 2 and the points at infinity $\infty=(0:\sqrt{\epsilon}:1:0)$, $\overline{\infty}=(0:- \sqrt{\epsilon}:1:0)$ are the $\F$-rational points of order 2 if $\epsilon$ is a square in $\F$. So 
\begin{equation*}
\J_{\epsilon,\delta}[2]=\{(0:\pm 1:0:1),(0:\pm \sqrt{\epsilon}:1:0) \}.
\end{equation*}

If $\chi(\epsilon) =1$ then the Jacobi curve $\chi(\epsilon) =1$ has full 2-torsion. In addition, the subfamily of Jacobi curves $\J_{\epsilon,\delta}$ for all $\epsilon,\delta \in \F$ with $\chi(\epsilon) =1$ is the family of all elliptic curves over $\F$ with full 2-torsion $\F$-rational points.

\begin{lemma}\label{LemJ1}
	For Jacobi quartic curve $\J_{\epsilon,\delta}$ over the field $\F$ of $\ch(\F) \ne 2$,  we have
	\begin{equation*}
	\label{J2}
	\J_{\epsilon,\delta}(\F)[2]= 
	\begin{cases}
	<\mathcal{T}> \quad \;\;\, \cong \mathbb{Z}_2,  &\text{if }  \chi(\epsilon)=-1,\\
	<\mathcal{T},\infty> \; \cong \mathbb{Z}_2\times \mathbb{Z}_2,  & \text{if } \chi(\epsilon)=1.\\
	\end{cases}
	\end{equation*}
\end{lemma}
\begin{proof}
	If $ \chi(\epsilon)=-1$ then $J[2](\F)=\{\cO,\mathcal{T}\}=<\mathcal{T}>$. And, if $\chi(\epsilon)=1$, then Jacobi quartic curve $\J_{\epsilon,\delta}$ has the full $2$-torsion points $\{\cO,\mathcal{T},\infty,\overline{\infty}\}$. Therefore, $J[2](\F)=J[2]= <\mathcal{T},\infty>$.
	\qed
\end{proof}

From Lemma \ref{LemJ1}, we see that, for the point $P=(x,y) \in  \J_{\epsilon,\delta}(\F) $, the coset $P+\J_{\epsilon,\delta}(\F)[2]$ of $\J_{\epsilon,\delta}(\F)[2]$, is given by
\begin{equation*}
P+\J_{\epsilon,\delta}(\F)[2]=
\begin{cases}
\{ (x,y),(-x,-y) \},  &\text{if }  \chi(\epsilon)=-1,\\
\{ (x,y),(-x,-y),(\frac{1}{\sqrt{\epsilon}x},\frac{-y}{\sqrt{\epsilon}x^2}),(\frac{-1}{\sqrt{\epsilon}x},\frac{y}{\sqrt{\epsilon}x^2}) \},  & \text{if } \chi(\epsilon)=1.\\
\end{cases}
\end{equation*}



Here, we compute the set of 4-torsion points of $\J_{\epsilon,\delta}$. If $P$ is a point of order 4, then we have $2P \in J[2]$, so  $2P=(0,-1)$  or $2P=(0:\pm \sqrt{\epsilon}:1:0)$. If $2P=(0,-1)$ then, from the doubling formula~\eqref{eq:DbEJ}, we have $P \in \{(\pm  \sqrt{r_1} , 0), (\pm  \sqrt{r_2} , 0)\}$, where 
\begin{equation*}
r_1=\dfrac{-\delta+\sqrt{\delta^2-\epsilon}}{\epsilon}~,~r_2=\dfrac{-\delta-\sqrt{\delta^2-\epsilon}}{\epsilon}.
\end{equation*}
If $2P=(0:\pm \sqrt{\epsilon}:1:0)$ then, from addation formula \eqref{eq:TEH1}, we have  
$XY \ne 0$, $T=\pm \frac{1}{\sqrt{\epsilon}}$, $X^2=\pm \frac{1}{\sqrt{\epsilon}}$ and $Y^2=2(1 \pm \frac{\delta}{\sqrt{\epsilon}} )$. So, 
$P \in \{(\pm \sqrt{\alpha_1}, \pm \sqrt{\beta_1} ), (\pm \sqrt{\alpha_2}, \pm \sqrt{ \beta_2} )\}$, where 

\begin{equation*}
\alpha_1=\frac{1}{\sqrt{\epsilon}}~,~\alpha_2=\frac{-1}{\sqrt{\epsilon}}~,~\beta_1 =2(1+\frac{\delta}{\sqrt{\epsilon}})~,~\beta_2 =2(1-\frac{\delta}{\sqrt{\epsilon}}).
\end{equation*}
Therefore, if Jacobi quartic curve $\J_{\epsilon,\delta}$  has a point $P$ of order $4$ then  
\begin{equation*}
P \in \left\{(\pm \sqrt{r_1} , 0 ),(\pm \sqrt{ r_2}, 0 ),(\pm \sqrt{\alpha_1}, \pm \sqrt{\beta_1} ), (\pm \sqrt{\alpha_2}, \pm \sqrt{ \beta_2} )\right\}.
\end{equation*}
For simplicity, we let 
$Q_1=(\sqrt{r_1} , 0 )$, $Q_2=(\sqrt{r_2}, 0)$,
$Q_3=(\sqrt{\alpha_1}, \sqrt{\beta_1})$, $\overline{Q}_3=(-\sqrt{\alpha_1}, -\sqrt{\beta_1})$, $Q_4=(\sqrt{\alpha_2},\sqrt{\beta_2})$ and  $\overline{Q}_4=(-\sqrt{\alpha_2},-\sqrt{\beta_2})$. 
So, 
\begin{equation*}
\J_{\epsilon, \delta}[4]=\left\{\cO,\mathcal{T},\infty,\overline{\infty}, \pm Q_1, \pm Q_2, \pm Q_3,  \pm \overline{Q}_3, \pm Q_4, \pm\overline{Q}_4 \right\}.
\end{equation*}
The group table of $\J_{\epsilon, \delta}[4]$ is easily obtained from Table~\ref{Tab:grp}.
\begin{table}[h!]
	\centering
	\caption{Group table of $\J_{\epsilon, \delta}$[4]}\label{Tab:grp}
	\renewcommand{\arraystretch}{2.25}
	\begin{tabular}{c|c|c|c|c|c|c|c}
		\hline
		$ + $ & $\mathcal{T}$&$ \infty$ &$\overline{\infty}$& $~ Q_1~$ &$ Q_2~$ &$ Q_3~$&$Q_4$\\
		\hline
		$Q_1$&$-Q_1$&$Q_2$&$-Q_2$&$\mathcal{T}$&$\overline{\infty}$&$Q_4$&$ \overline{Q}_3$ \\
		\hline
		$Q_2$&$-Q_2$&$Q_1$&$-Q_1$&$\overline{\infty}$&$\mathcal{T}$&$- \overline{Q}_4$ &$- \overline{Q}_3 $ \\
		\hline
		$Q_3$&$ \overline{Q}_3 $&$-Q_3$&$-\overline{Q}_3 $&$Q_4$&$- \overline{Q}_4$&$\infty $&$Q_2$ \\
		\hline
		$Q_4$&$ \overline{Q}_4$ &$- \overline{Q}_4 $&$-Q_4$&$ \overline{Q}_3$&$- \overline{Q}_3$&$Q_2$& $\overline{\infty} $\\
		\hline
	\end{tabular}
\end{table}
\begin{lemma}\label{LemJ3}
For Jacobi quartic curve $\J_{\epsilon,\delta}$ over the field $\F$ of $\ch(\F) \ne 2$,
if $\epsilon$ is a non-square in $\F$ then we have 
\begin{equation*}
\label{J2}
J[4](\F)= 
\begin{cases}
<\mathcal{T}>, &\text{if }  \chi(\delta^2-\epsilon)=-1,\\
<\mathcal{T}>, & \text{if }  \chi(\delta^2-\epsilon)=1, \chi(r_1)=\chi(r_2)=-1, \\
<Q_1> or <Q_2>, & \chi(\delta^2-\epsilon)=1, \chi(r_1)=1 \text{ or } \chi(r_2)=1.\\
\end{cases}
\end{equation*}
\end{lemma}
\begin{proof}
Since $\chi(\epsilon)=-1$, so the only points of order $2$ are $(0,\pm 1)$ and $\alpha_1, \alpha_2 \notin \F $. So  $Q_3 , Q_4 \notin \J[4](\F) $.

\begin{description}
\item[$\bullet$] If $\chi(\delta^2-\epsilon)=-1$ then $r_1, r_2 \notin \F $, so  $Q_1 , Q_2 \notin \J[4](\F) $. Hence 
$J[4](\F)=J[2](\F)=\{\cO,\mathcal{T} \}$, so
$ J[4](\F) = <\mathcal{T}> $.

\item[$\bullet$] If $\chi(\delta^2-\epsilon)=1$  then $r_1, r_2 \in \F $. 
 Since $\chi(r_1 r_2)=\chi(\epsilon)$ and $\chi(\epsilon)=-1$, so either both of the $r_1$ and $r_2$ are non-square or exactly one of them is square. So we need to check two cases, $\chi(r_1)=\chi(r_2)=-1$ and $\chi(r_1)=1$ or $\chi(r_2)=1$.

\begin{itemize}
	\item Let $\chi(r_1)=\chi(r_2)=-1$. 
 so  $Q_1 , Q_2 \notin \J[4](\F) $. Hence 
$J[4](\F)=J[2](\F)=\{\cO,\mathcal{T} \}$, so $ J[4](\F) = <\mathcal{T}> $.

\item Let $\chi(r_1)=1$ or $\chi(r_2)=1$.  If $\chi(r_1)=1$ then $\chi(r_2)=-1$ and if $\chi(r_2)=1$ then $\chi(r_1)=-1$. So exactly one of the points of  $Q_1$ and $Q_2$ are belongs to $\J[4](\F)$ and other is not. So $ J[4](\F) = <Q_1> $ or $ J[4](\F) = <Q_2> $. \qed
\end{itemize}
	\end{description}
\end{proof}
\begin{lemma}\label{LemJ3}
For Jacobi quartic curve $\J_{\epsilon,\delta}$ over the field $\F$ of $\ch(\F) \ne 2$,
if $\epsilon$ is a square in $\F$ then we have 
\begin{itemize}
\item[(i)] If $\chi(\delta^2-\epsilon)=-1$  then  
\begin{equation*}
		\begin{array}{c}
			\label{J4}
			J[4](\F) = \\
			\begin{cases}
<\mathcal{T},\infty >, &\text{if }   \chi(-1)=-1, \chi(\sqrt{\epsilon})=\chi(-\sqrt{\epsilon})=-1,\\
<\mathcal{T},\infty >, &\text{if }   \chi(-1)=-1, \chi(\sqrt{\epsilon})=1, \chi(\beta_1)=-1,\\
<\mathcal{T},\infty >, &\text{if }  \chi(-1)=-1, \chi(-\sqrt{\epsilon})=1, \chi(\beta_2)=-1,\\
 <Q_{3},\mathcal{T}>, & \text{if }    \chi(-1)=-1, \chi(\sqrt{\epsilon})=\chi(\beta_1)=1, \\
<Q_{4},\mathcal{T}>, & \text{if }  \chi(-1)=-1, \chi(-\sqrt{\epsilon})=\chi(\beta_2)=1, \\
<\mathcal{T},\infty>, &\text{if }   \chi(-1)=1, \chi(\sqrt{\epsilon})=-1 , \\
<\mathcal{T}, \infty>, &\text{if }   \chi(-1)=1, \chi(\sqrt{\epsilon})=1, \chi(\beta_1)=\chi(\beta_2)=-1, \\
<Q_{3},\mathcal{T}> , &\text{if }   \chi(-1)=1, \chi(\sqrt{\epsilon})=1, \chi(\beta_1)=1, \\
<Q_{4},\mathcal{T}> , &\text{if }   \chi(-1)=1, \chi(\sqrt{\epsilon})=1, \chi(\beta_2)=1,\\
			\end{cases}
		\end{array}
	\end{equation*}
\item[(ii)] If $\chi(\delta^2-\epsilon)=1$  then  
\begin{equation*}
		\begin{array}{c}
			\label{J4}
			J[4](\F) =\\
		\begin{cases}
<Q_{1},\infty> =<Q_{2},\infty>,   &\text{if }   \chi(-1)=-1,  \chi(r_1)=1,\\
<\mathcal{T},\infty >, &\text{if }   \chi(-1)=-1,\chi(r_1)=-1, \chi(\sqrt{\epsilon})=\chi(-\sqrt{\epsilon})=-1, \\
<\mathcal{T},\infty >, &\text{if }  \chi(-1)=-1,\chi(r_1)=-1, \chi(\sqrt{\epsilon})=1, \chi(\beta_1)=-1, \\
<\mathcal{T},\infty >, &\text{if }   \chi(-1)=-1,\chi(r_1)=-1,  \chi(-\sqrt{\epsilon})=1, \chi(\beta_2)=-1,\\
 <Q_{3},\mathcal{T}>, & \text{if }   \chi(-1)=-1, \chi(r_1)=-1, \chi(\sqrt{\epsilon})=\chi(\beta_1)=1, \\
<Q_{4},\mathcal{T}>, & \text{if }    \chi(-1)=-1, \chi(r_1)=-1,\chi(-\sqrt{\epsilon})=\chi(\beta_2)=1, \\
<\mathcal{T},\infty>,  &\text{if }   \chi(-1)=1,  \chi(r_1)=-1,\\
<Q_{1},\infty> =<Q_{2},\infty>, &\text{if }  \chi(-1)=1, \chi(r_1)=1,  \chi(\sqrt{\epsilon})=-1,\\
J[4]=<Q_1,Q_3>, &\text{if }   \chi(-1)=1, \chi(r_1)=1, \chi(\sqrt{\epsilon})=1.
				
			\end{cases}
		\end{array}
	\end{equation*}
\end{itemize}
\end{lemma}
\begin{proof}
Since $\chi(\epsilon)=1$, so Jacobi quartic curve $\J_{\epsilon,\delta}$ has a full 2-torsion subgroup as 
\begin{equation*}
J[2](\F)=\{\cO,\mathcal{T},\infty,\overline{\infty} \}.
\end{equation*}
We knowe that, if Jacobi quartic curve $\J_{\epsilon,\delta}$  has a point $P$ of order $4$ then  $2P \in J[2](\F)$ and  
\begin{equation*}
P \in \{(\pm \sqrt{r_1} , 0 ),(\pm \sqrt{ r_2}, 0 ),(\pm \sqrt{\alpha_1}, \pm \sqrt{\beta_1} ), (\pm \sqrt{\alpha_2}, \pm \sqrt{ \beta_2} )\}.
\end{equation*}
We have also
\begin{equation*}
\chi(r_1r_2)=\chi(\epsilon)~~,~~\chi(\alpha_1 \alpha_2)=\chi(-\epsilon)~~,~~\chi(\beta_1\beta_2)=\chi(-\epsilon(\delta^2-\epsilon)).
\end{equation*}

\begin{itemize}
\item[(i)] If $\chi(\delta^2-\epsilon)=-1$ then $r_1,r_2 \notin \F$ and so $Q_1,Q_2 \notin \J[4](\F)$. So from group table \eqref{Tab:grp} either none of the $Q_3$ and $Q_4$ belongs to $\J[4](\F)$  or 
only one of them is  belongs to $\J[4](\F)$. We need to check two cases, $\chi(-1)=1$ and $\chi(-1)=-1$.
\begin{description}
 \item[$\bullet$] Let $\chi(-1)=-1$. Since  $ \chi(\alpha_1 \alpha_2)=\chi(-\epsilon)=-1$, 
 so either both of the $\alpha_1$ and $\alpha_2$ are non-square or exactly one of them is square.
\begin{itemize}
\item  If $\chi(\sqrt{\epsilon})=\chi(-\sqrt{\epsilon})=-1$ then  $\alpha_1 , \alpha_2 \notin \F$, so there is no point of order $4$ and $J[4](\F)=J[2](\F)$. So $ J[4](\F) = <\mathcal{T},\infty> $.
\item If $\chi(\sqrt{\epsilon})=1$  then $\chi(\alpha_1)=1$ and so $\chi(\alpha_2)=-1$. So $Q_4 \notin \J[4](\F)$. 
\begin{itemize}
\item[-] If  $\chi(\beta_1)=-1$ then $Q_3 \notin \J[4](\F)$. So
there is no point of order $4$ and $J[4](\F)=J[2](\F)$. So $ J[4](\F) = <\mathcal{T},\infty> $.
\item[-] If  $\chi(\beta_1)=1$ then $Q_3 \in \J[4](\F)$. So  $ J[4](\F) = <Q_3,\mathcal{T}> $.
\end{itemize}
\item If $\chi(-\sqrt{\epsilon})=1$  then  $\chi(\alpha_2)=1$ and so $\chi(\alpha_1)=-1$. So $Q_3 \notin \J[4](\F)$. 
\begin{itemize}
\item[-] If  $\chi(\beta_2)=-1$ then $Q_4 \notin \J[4](\F)$. So
there is no point of order $4$ and $J[4](\F)=J[2](\F)$. So $ J[4](\F) = <\mathcal{T},\infty> $.
\item[-] If  $\chi(\beta_2)=1$ then $Q_4 \in \J[4](\F)$. So  $ J[4](\F) = <Q_4,\mathcal{T}> $.
\end{itemize}
\end{itemize}
\item \item[$\bullet$] Let $\chi(-1)=1$. Since  $\chi(\alpha_1 \alpha_2)=\chi(-\epsilon)=1$, so $\chi(\sqrt{\epsilon})=\chi(-\sqrt{\epsilon})$.  
\begin{itemize}
	\item  If $\chi(\sqrt{\epsilon})=-1$ then  $\chi(-\sqrt{\epsilon})=-1$, so  $\alpha_1 , \alpha_2 \notin \F$. Hence, there is no point of order $4$ and $J[4](\F)=J[2](\F)$. So $ J[4](\F) = <\mathcal{T},\infty> $.

\item  If $\chi(\sqrt{\epsilon})=1$ then  $\chi(-\sqrt{\epsilon})=1$, so  $\alpha_1 , \alpha_2 \in \F$.
Also, since  $ \chi(\delta^2- \epsilon)=-1$ and $\chi(-1)=1$ so we have  $\chi(\beta_1 \beta_2)=-1$. 
So either both of the $\beta_1$ and $\beta_2$ are non-square or exactly one of them is square.
\begin{itemize}
\item[-] If $\chi(\beta_1)=\chi(\beta_2)=-1$ then there is no point of order $4$ and $J[4](\F)=J[2](\F)$. So $ J[4](\F) = <\mathcal{T},\infty> $.
\item[-] If $\chi(\beta_1)=1$ then $\chi(\beta_2)=-1$. Hence $Q_3 \in \J[4](\F)$ and $Q_4 \notin \J[4](\F)$. So  $ J[4](\F) = <Q_3,\mathcal{T}> $.
\item[-] If $\chi(\beta_2)=1$ then $\chi(\beta_1)=-1$. Hence $Q_3 \notin \J[4](\F)$ and $Q_4 \in \J[4](\F)$. So  $ J[4](\F) = <Q_4,\mathcal{T}> $.
 \end{itemize}
 \end{itemize}
\end{description}
\begin{itemize}
\item[(ii)] If $\chi(\delta^2-\epsilon)=1$ then $r_1, r_2 \in \F$.
Since $\chi(r_1 r_2)=\chi(\epsilon)$ and $\chi(\epsilon)=1$, so both of the $r_1$ and $r_2$ are square or both of them not square. We need to check two cases, $\chi(-1)=1$ and $\chi(-1)=-1$.
\begin{description}
\item	\item[$\bullet$] Let $\chi(-1)=-1$. 
So, $\chi(\alpha_1 \alpha_2)=\chi(-\epsilon)=-1$ and $\chi(\beta_1 \beta_2)=\chi(-\epsilon(\delta^2-\epsilon))=-1$. 
So, either none of the $Q_3$ and $Q_4$ belongs to $\J[4](\F)$  or 
only one of them is  belongs to $\J[4](\F)$.
	\begin{itemize}
		\item[$-$]  If $\chi(r_1)=1$ then $\chi(r_2)=1$, so $Q_1,Q_2 \in J[4](\F)$. 
	In addation, $Q_3 , Q_4 \notin J[4](\F)$.  Since $J[4](\F)$ has a group stracture, so if $Q_3 \in J[4](\F)$ then $Q_1+Q_3=Q_4 \in J[4](\F) $ where is a contradiction. 
Also, if $Q_4 \in J[4](\F)$  then $Q_1+Q_4= \overline{Q}_3 \in J[4](\F)$
and again is a contradiction. So  $ J[4](\F) =<Q_1,\infty>=<Q_2,\infty> $.

\item[$-$] If $\chi(r_1)=-1$ then  $\chi(r_2)=-1$. So $Q_1,Q_2 \notin J[4](\F)$. 
Since $\chi(-\epsilon)=-1$, so either both of the $\sqrt{\epsilon}$ and $-\sqrt{\epsilon}$ are non-square or exactly one of them is square.
\begin{itemize}
\item[-] If $\chi(\sqrt{\epsilon})=\chi(-\sqrt{\epsilon})=-1$ then $Q_3,Q_4 \notin J[4](\F)$. So there is no point of order $4$ and $ J[4](\F)=<\mathcal{T},\infty> $. 
\item[-] If $\chi(\sqrt{\epsilon})=1$ then $\chi(-\sqrt{\epsilon})=-1$. So $Q_4 \notin J[4](\F)$. Now if $\chi(\beta_1)=-1$ then  $Q_3 \notin J[4](\F)$. So there is no point of order $4$ and $ J[4](\F)=<\mathcal{T},\infty> $. On the other hand, if $\chi(\beta_1)=1$ then  $Q_3 \in J[4](\F)$. So $ J[4](\F) = <Q_3,\mathcal{T}> $.
\item[-] If $\chi(-\sqrt{\epsilon})=1$ then $\chi(\sqrt{\epsilon})=-1$. So $Q_3 \notin J[4](\F)$. Now if $\chi(\beta_2)=-1$ then  $Q_4 \notin J[4](\F)$. So there is no point of order $4$ and $ J[4](\F)=<\mathcal{T},\infty> $. On the other hand, if $\chi(\beta_2)=1$ then  $Q_4 \in J[4](\F)$. So $ J[4](\F) = <Q_4,\mathcal{T}> $.
\end{itemize}
	\end{itemize}
	\item \item[$\bullet$] Let $\chi(-1)=1$. So $\chi(\alpha_1 \alpha_2)=\chi(-\epsilon)=1$ and $\chi(\beta_1 \beta_2)=\chi(-\epsilon(\delta^2-\epsilon))=1$. So $\chi(\sqrt{\epsilon})=\chi(-\sqrt{\epsilon})$ and $\chi(\beta_1)=\chi( \beta_2)$.  We need to check two cases, $\chi(r_1)=1$ and $\chi(r_1)=-1$. 
	\begin{itemize}
	\item[$-$]  If $\chi(r_1)=-1$ then  $\chi(r_2)=-1$. So $Q_1,Q_2 \notin J[4](\F)$. If $\chi(\sqrt{\epsilon})=-1$ then  $\chi(-\sqrt{\epsilon})=-1$.
So there is no point of order $4$ and  $ J[4](\F)=<\mathcal{T},\infty> $. On the other hand,
 if $\chi(\sqrt{\epsilon})=1$ then $\chi(-\sqrt{\epsilon})=1$. In this case $\chi(\beta_1)=\chi(\beta_2)=-1$ So there is no point of order $4$ and  $ J[4](\F)=<\mathcal{T},\infty> $. Note  that if  $\chi(\beta_1)=\chi(\beta_2)=1$ then  $Q_3,Q_4 \in J[4](\F)$. But $J[4](\F)$ has a group stracture, so  $Q_3+Q_4=Q_1 \in J[4](\F)$,  which is a contradiction. 
\item[$-$] If $\chi(r_1)=1$ then $\chi(r_2)=1$. Hence
	 $Q_1,Q_2 \in J[4](\F).$
	\begin{description}
	\item[-] If $\chi(\sqrt{\epsilon})=-1$ then $\chi(-\sqrt{\epsilon})=-1$. So $Q_3,Q_4 \notin J[4](\F)$.
So $ J[4](\F) =<Q_1,\infty>=<Q_2,\infty> $.
\item[-] If $\chi(\sqrt{\epsilon})=1$ then $\chi(-\sqrt{\epsilon})=1$. So $\chi(\alpha_1)=\chi(\alpha_2)=1$.  Since $\chi(r_1)=\chi(r_2)=1$,  we can write $\alpha_1 \beta_1=-(\sqrt{r_1}+\sqrt{r_2})^2$, so $\chi(\alpha_1 \beta_1) =\chi(-1)=1$.  Hence  $\chi(\alpha_1 ) =\chi(\beta_1)=1$ and so $\chi(\alpha_2) =\chi(\beta_2)=1$. 
So Jacobi quartic curve $\J_{\epsilon, \delta}$ has the full $4$-torsion subgroup, so $ J[4](\F)=J[4]= <Q_1,Q_3>$. \qed
    \end{description}
	\end{itemize}
\end{description}
\end{itemize}
\end{itemize}
\end{proof}

\begin{lemma}\label{LemJ3}
For Jacobi quartic curve $\J_{\epsilon,\delta}$ over the finite field $\F_q$ of odd characteristic, if $\chi(\epsilon)=-1$ then we have 
\begin{equation*}
\label{J2}
J[4](\F_q)= 
\begin{cases}
<\mathcal{T}>, &\text{if }  \chi(\delta^2-\epsilon)=-1,\\
<Q_1> or <Q_2>, & \text{if } \chi(\delta^2-\epsilon)=1.\\
\end{cases}
\end{equation*}
\end{lemma}
\begin{proof}
Assume that Jacobi quartic curve $\J_{\epsilon,\delta}$ has a point of order 4. Since the points at infinity in the nonsingular model of the Jacobi curve are 2 torsion points, so there exist an affine point $P=(x,y)$ of order 4 and so, $2P \in J[2](\F_q)$. Since $\chi(\epsilon)=-1 $, so we have $2P=(0,1)$ or $(0,-1)$. From addition formula \eqref{eq:AdEJ},$xy=0$. If $x=0$ then $P$ equals $(0, \pm 1)$ which is a $2$- torsion point. So, $y=0$ and then we have $\epsilon x^4+2\delta x^2+1=0$. Thus, $\chi(\delta^2-\epsilon)=1$.
So, if the Jacobi quartic curve $\J_{\epsilon,\delta}$ over $\F_q$ of odd characteristic has a point of order 4  and $\chi(\epsilon)=-1$ then, we have $\chi(\delta^2-\epsilon)=1$.
 So, if $\chi(\delta^2-\epsilon)=-1$ then  there is no point of order $4$ and  
 $J[4](\F_q)=J[2](\F_q)=\{\cO,\mathcal{T} \}$. So $ J[4](\F_q) =<\mathcal{T}> $ . 

On the other hand, if  $\chi(\delta^2-\epsilon)=1$ then $r_1,r_2 \in \F_q$. Since $\chi(r_1r_2)=\chi(\epsilon)=-1$, so exactly one of the  $r_1$ and $r_2$ is square and the other is not square. So
\begin{equation*}
J[4]=\{\cO,\mathcal{T},\pm Q_1\},~~or~~ J[4]=\{\cO,\mathcal{T},\pm Q_2\}.
\end{equation*}
So,
$ J[4](\F_q) =<Q_1>$  or  $ J[4](\F_q) =<Q_2>$ . \qed
\end{proof}

\begin{lemma}\label{LemJ4}
For Jacobi quartic curve $\J_{\epsilon,\delta}$ over the finite field $\F_q$ of odd characteristic,  if $\chi(\epsilon)=1$ then we have

\begin{equation*}
\begin{array}{c}
\label{J4}
J[4](\F_q) = \\
$~$\\
\begin{cases}
<\mathcal{T},\infty >, &\text{if }     \chi(\delta^2- \epsilon)=-1, \chi(-1)=-1, \chi(\beta_1)=-1, \\
 <Q_{3},\mathcal{T}> or  <Q_{4},\mathcal{T}>, & \text{if }  \chi(\delta^2- \epsilon)=-1, \chi(-1)=-1, \chi(\beta_1)=1, \\
<\mathcal{T},\infty>, &\text{if }    \chi(\delta^2- \epsilon)=-1,\chi(-1)=1, \chi(\sqrt{\epsilon})=-1 ,  \\
<Q_{3},\mathcal{T}> or  <Q_{4},\mathcal{T}>, &\text{if }    \chi(\delta^2- \epsilon)=-1, \chi(-1)=1, \chi(\sqrt{\epsilon})=1 ,   \\
<Q_{1},\infty> =<Q_{2},\infty>,   &\text{if }   \chi(\delta^2- \epsilon)=1,\chi(-1)=-1,  \chi(r_1)=1,\\
<Q_{3},\mathcal{T}> or  <Q_{4},\mathcal{T}>, &\text{if }   \chi(\delta^2- \epsilon)=1,  \chi(-1)=-1,\chi(r_1)=-1,\\
<\mathcal{T},\infty>,  &\text{if }   \chi(\delta^2- \epsilon)=1, \chi(-1)=1,  \chi(r_1)=-1, \\
<Q_{1},\infty> =<Q_{2},\infty>, &\text{if }   \chi(\delta^2- \epsilon)=1, \chi(-1)=1, \chi(r_1)=1,  \chi(\alpha_1)=-1 , \\
J[4]=<Q_1,Q_3>, &\text{if }   \chi(\delta^2- \epsilon)=1, \chi(-1)=1, \chi(r_1)=1, \chi(\alpha_1)=1,

\end{cases}
\end{array}
\end{equation*}
\end{lemma}
\begin{proof}
Since $\chi(\epsilon)=1$, so Jacobi quartic curve $\J_{\epsilon,\delta}$ has a full 2-torsion subgroup as 
\begin{equation*}
J[2](\F_q)=\{\cO,\mathcal{T},\infty,\overline{\infty} \}.
\end{equation*}
We knowe that, if Jacobi quartic curve $\J_{\epsilon,\delta}$  has a point $P$ of order $4$ then  $2P \in J[2](\F_q)$ and  
\begin{equation*}
P \in \{(\pm \sqrt{r_1} , 0 ),(\pm \sqrt{ r_2}, 0 ),(\pm \sqrt{\alpha_1}, \pm \sqrt{\beta_1} ), (\pm \sqrt{\alpha_2}, \pm \sqrt{ \beta_2} )\}.
\end{equation*}
We have also
\begin{equation*}
\chi(r_1r_2)=\chi(\epsilon)~~,~~\chi(\alpha_1 \alpha_2)=\chi(-\epsilon)~~,~~\chi(\beta_1\beta_2)=\chi(\epsilon(\epsilon-\delta^2)).
\end{equation*}

\noindent{\textbf{Case $\delta^2-\epsilon$ quadratic non-residue:}}
If $\chi(\delta^2-\epsilon)=-1$ then  $r_1, r_2 \notin \F_q $, so  $Q_1 , Q_2 \notin \J[4](\F_q) $. We need to check two cases, $\chi(-1)=1$ and $\chi(-1)=-1$.

\begin{description}
 \item[$\bullet$] Let $\chi(-1)=-1$. Since  $ \chi(\delta^2- \epsilon)=-1$ and $\chi(-1)=-1$ we have  $\chi(\beta_1 \beta_2)=\chi(\epsilon(\epsilon- \delta^2))=1$.
\begin{itemize}
\item  If $\chi(\beta_1)=-1$ then  $\chi(\beta_2)=-1$.  So there is no point of order $4$ and 
 $J[4](\F_q)=J[2](\F_q)$. So $ J[4](\F_q) = <\mathcal{T},\infty> $.
\item If $\chi(\beta_1)=1$ then  $\chi(\beta_2)=1$. Also we have $\chi(\alpha_1 \alpha_2)=\chi(-\epsilon)=-1$, so one of the $\alpha_1$ and $\alpha_2$ is square and other is not. So,  
$Q_3 \in J[4](\F_q)$ ~or~$Q_4 \in J[4](\F_q)$.
 Hence $ J[4](\F_q) =<Q_{3},\mathcal{T}>$ or $ J[4](\F_q) =<Q_{4},\mathcal{T}>$ .
\end{itemize}
\item \item[$\bullet$] Let $\chi(-1)=1$. Since  $ \chi(\delta^2- \epsilon)=-1$ and $\chi(-1)=1$ we have  $\chi(\beta_1 \beta_2)=-1$.  Also we have $\chi(\alpha_1 \alpha_2)=\chi(-\epsilon)=1$. 
\begin{itemize}
	\item  If $\chi(\alpha_1)=-1$  then $\chi( \alpha_2)=-1$.  So there is no point of order $4$ and $J[4](\F_q)=J[2](\F_q)$. So $ J[4](\F_q) =<\mathcal{T},\infty> $.
	\item If $\chi(\alpha_1)=1$ then $\chi(\alpha_2)=1$. Also, $\chi(\beta_1 \beta_2)=-1$, so one of the 
$\beta_1$ and $  \beta_2$ is square and other is not.
So, 
$Q_3 \in J[4](\F_q)$ ~or~$Q_4 \in J[4](\F_q)$.
 Hence $ J[4](\F_q) =<Q_{3},\mathcal{T}>$ or $ J[4](\F_q) =<Q_{4},\mathcal{T}>$ .
 \end{itemize}
\end{description}
 \noindent{\textbf{Case $\delta^2-\epsilon$ quadratic residue:}} If $\chi(\delta^2-\epsilon)=1$ then  $r_1, r_2 \in \F_q $.
Since $\chi(r_1 r_2)=\chi(\epsilon)$ and $\chi(\epsilon)=1$, so both of the $r_1$ and $r_2$ are square or both of them not square. We need to check two cases, $\chi(-1)=1$ and $\chi(-1)=-1$.
\begin{description}
	\item[$\bullet$] Let $\chi(-1)=-1$. 
So, $\chi(\alpha_1 \alpha_2)=\chi(-\epsilon)=-1$ and $\chi(\beta_1 \beta_2)=\chi(\epsilon(\epsilon-\delta^2))=-1$. 
 So, $\chi(\alpha_1 \beta_1) =\chi(\alpha_2 \beta_2)$. 
	\begin{itemize}
		\item  If $\chi(r_1)=1$ then $\chi(r_2)=1$, so
		$Q_1,Q_2 \in J[4](\F_q)$.
In addation, if $\chi(\alpha_1 \beta_1) =1$ then $\chi(\alpha_2 \beta_2) =1$,  and  exactly one of the $Q_3$ and $Q_4$
 is belong to $J[4](\F_q)$ and other is not.
But $J[4](\F_q)$ has a group stracture, so if $Q_3 \in J[4](\F_q)$
 then $Q_1+Q_3=Q_4 \in J[4](\F_q)$
  where is a contradiction. 
Also, if $Q_4 \in J[4](\F_q)$ 
then $Q_1+Q_4= \overline{Q}_3 \in J[4](\F_q)$
and again is a contradiction. Hence, in this case $\chi(\alpha_1 \beta_1) = \chi(\alpha_2 \beta_2) = -1$ . So  $ J[4](\F_q) =<Q_1,\infty>=<Q_2,\infty> $.

		\item If $\chi(r_1)=-1$ then  $\chi(r_2)=-1$. So $\chi(-r_1)=\chi(-r_2)=1$. Then we can write $\alpha_1 \beta_1=(\sqrt{-r_1}+\sqrt{-r_2})^2$, so $\chi(\alpha_1 \beta_1) =1$.  Then  exactly one of the 
$Q_3$ and $Q_4$ is belong to $J[4](\F_q)$ and other is not. So   $ J[4](\F_q) =<Q_3,\mathcal{T}>$ or $ J[4](\F_q) =<Q_4,\mathcal{T}> $. 
	\end{itemize}
	\item[$\bullet$] Let $\chi(-1)=1$. So, $\chi(\alpha_1 \alpha_2)=\chi(-\epsilon)=1$ and $\chi(\beta_1 \beta_2)=\chi(\epsilon(\epsilon-\delta^2))=1$. 
	\begin{itemize}
	\item  If $\chi(r_1)=-1$ then  $\chi(r_2)=-1$. Hence $Q_1,Q_2 \notin J[4](\F_{q}).$ 
\begin{description} 
\item If $\chi(\beta_1)=-1$ then $\chi(\beta_2)=-1$. So there is no point of order $4$ and  $J[4](\F_q)=J[2](\F_q)$. So $ J[4](\F_q)=<\mathcal{T},\infty> $.
\item If $\chi(\beta_1)=1$ then $\chi(\beta_2)=1$. So, If $\chi(\alpha_1)=\chi(\alpha_2)=-1$ then there is no point of order $4$ and  $J[4](\F_q)=J[2](\F_q)$. So $ J[4](\F_q) =<\mathcal{T},\infty> $. On the other hand,  If $\chi(\alpha_1)=\chi(\alpha_2)=1$ then  
$Q_3,Q_4 \in J[4](\F_q)$.
But $J[4](\F_q)$ has a group stracture, so 
$Q_3+Q_4=Q_1 \in J[4](\F_q)$.
 which is a contradiction. So, in this case,  if $\chi(\beta_1)=1$ then $\chi(\alpha_1)=\chi(\alpha_2)=-1$, and so,  $ J[4](\F_q) =<\mathcal{T},\infty> $.
 \end{description}
	\item If $\chi(r_1)=1$ then $\chi(r_2)=1$. Hence
	 $Q_1,Q_2 \in J[4](\F_{q}).$
	\begin{description}
	\item  If $\chi(\alpha_1)=-1$ then $\chi(\alpha_2)=-1$. So $ J[4](\F_q) =<Q_1,\infty>=<Q_2,\infty> $.
	\item If $\chi(\alpha_1)=1$ then $\chi(\alpha_2)=1$. We know that $\chi(\beta_1 \beta_2)=1$, so $\chi(\beta_1)=\chi(\beta_2)$. Since $\chi(r_1)=\chi(r_2)=1$,  we can write $\alpha_1 \beta_1=-(\sqrt{r_1}+\sqrt{r_2})^2$, so $\chi(\alpha_1 \beta_1) =\chi(-1)=1$.  Hence  $\chi(\alpha_1 ) =\chi(\beta_1)=1$ and so $\chi(\alpha_2) =\chi(\beta_2)=1$. 
So Jacobi quartic curve $\J_{\epsilon, \delta}$ has the full $4$-torsion subgroup, so $ J[4](\F_q)=J[4]= <Q_1,Q_3>$. \qed
    \end{description}
	\end{itemize}
\end{description}
\end{proof}
\begin{corollary}
For Jacobi quartic curve $\J_{\epsilon,\delta}$ over the finite field $\F_q$ of odd characteristic, we have
\begin{equation*}
\label{J4}
J[4](\F_q)\cong
\begin{cases}
\mathbb{Z}_2, &\text{if }  \chi(\epsilon)=-1, \chi(\delta^2- \epsilon)=-1,\\
\mathbb{Z}_4, &\text{if }   \chi(\epsilon)=-1 , \chi(\delta^2- \epsilon)=1, \\
\mathbb{Z}_2\times \mathbb{Z}_2, &\text{if }   \chi(\epsilon)=1 , \chi(\delta^2- \epsilon)=-1,  \chi(-1)=-1, \chi(\beta_1)=-1, \\
\mathbb{Z}_2\times \mathbb{Z}_2, &\text{if }   \chi(\sqrt{\epsilon})=-1 , \chi(\delta^2- \epsilon)=-1,\chi(-1)=1,\\
\mathbb{Z}_2\times \mathbb{Z}_2, &\text{if }   \chi(\epsilon)=1 , \chi(\delta^2- \epsilon)=1, \chi(-1)=1,  \chi(r_1)=-1,  \\
\mathbb{Z}_2\times \mathbb{Z}_4, &\text{if }   \chi(\epsilon)=1 , \chi(\delta^2- \epsilon)=-1, \chi(-1)=-1, \chi(\beta_1)=1, \\
\mathbb{Z}_2\times \mathbb{Z}_4, &\text{if }   \chi(\sqrt{\epsilon})=1 , \chi(\delta^2- \epsilon)=-1,  \chi(-1)=1, \\
\mathbb{Z}_2\times \mathbb{Z}_4, &\text{if }  \chi(\sqrt{\epsilon})=-1 , \chi(\delta^2- \epsilon)=1,  \chi(-1)=1,\chi(r_1)=1, \\
\mathbb{Z}_2\times \mathbb{Z}_4, &\text{if }   \chi(\epsilon)=1 , \chi(\delta^2- \epsilon)=1,  \chi(-1)=-1, \\
\mathbb{Z}_4\times \mathbb{Z}_4, &\text{if }   \chi(\sqrt{\epsilon})=1 , \chi(\delta^2- \epsilon)=1, \chi(-1)=1,   \chi(r_1)=1, \chi(\beta_1)=1. 
\end{cases}
\end{equation*}
\end{corollary}
If Jacobi quartic curve $ \J_{\epsilon,\delta}$ has a full $4$-torsion group then whose coset with respect to $P=(x,y)$ is as fallow:
\begin{equation*}
\begin{array}{l}
J[4]+P=\\
\{(x,y),(-x,-y),(1/\sqrt{\epsilon}x,-y/\sqrt{\epsilon}x^2),(-1/\sqrt{\epsilon}x,y/\sqrt{\epsilon}x^2),
(x_1,y_1),(-x_1,-y_1),\\
(x_2,y_2),(-x_2,-y_2),(x_3,y_3),(-x_3,-y_3), (x_4,y_4),(-x_4,-y_4),
(\bar{x}_3,\bar{y}_3),(-\bar{x}_3,-\bar{y}_3),\\
(\bar{x}_4,\bar{y}_4),(-\bar{x}_4,-\bar{y}_4)
 \}.
\end{array}
\end{equation*}
where 
\begin{equation*}
\begin{array}{l}
(x_1,y_1)=P+Q_1=(\frac{\sqrt{r_1}y}{1-\epsilon r_1\, x^2},\frac{2\sqrt{r_1(\delta^2-\epsilon)}x}{1-\epsilon r_1\, x^2}),\ \\
\vspace*{3pt}\\
(x_2,y_2)=P+Q_2=(\frac{\sqrt{r_2}y}{1-\epsilon r_2\, x^2},\frac{2\sqrt{r_2(\delta^2-\epsilon)}x}{1-\epsilon r_2\, x^2}) ,\\
\vspace*{3pt}\\
(x_3,y_3)=P+Q_3=(\frac{\sqrt{\beta_1}\,x+\sqrt{\alpha_1}\,y}{1-\sqrt{\epsilon}\, x^2},-\sqrt{\alpha_1\beta_1}
\frac{(1+\sqrt{\epsilon} \,x^2)}{\sqrt{\beta_1}\,x-\sqrt{\alpha_1}\,y}) ,\ \\
\vspace*{3pt}\\
(x_4,y_4)=P+Q_4=(\frac{\sqrt{\beta_2}\,x+\sqrt{\alpha_2}\,y}{1+\sqrt{\epsilon}\, x^2},-\sqrt{\alpha_2\beta_2}
\frac{(1-\sqrt{\epsilon} \,x^2)}{\sqrt{\beta_2}\,x+\sqrt{\alpha_2}\,y}) , \\
\vspace*{3pt}\\
(\bar{x}_3,\bar{y}_3)=P-Q_3=(\frac{\sqrt{\beta_1}\,x-\sqrt{\alpha_1}\,y}{1-\sqrt{\epsilon}\, x^2},\sqrt{\alpha_1\beta_1}
\frac{(1+\sqrt{\epsilon} \,x^2)}{\sqrt{\beta_1}\,x+\sqrt{\alpha_1}\,y}) ,\ \\
\vspace*{3pt}\\
(\bar{x}_4,\bar{y}_4)=P-Q_4=(\frac{\sqrt{\beta_2}\,x-\sqrt{\alpha_2}\,y}{1+\sqrt{\epsilon}\, x^2},\sqrt{\alpha_2\beta_2}
\frac{(1-\sqrt{\epsilon} \,x^2)}{\sqrt{\beta_2}\,x-\sqrt{\alpha_2}\,y}). \\
\end{array}
\end{equation*}

\begin{proposition}\label{JQTE}
Let $\F_q$ be a field with odd characteristic and $\epsilon,\delta \in \F_q$ such that $\epsilon(\delta^2-\epsilon)\ne 0$. If  $\chi(\delta^2-\epsilon) =1$ then  the subfamily of Jacobi curves $\J_{\epsilon,\delta}$  is the family of all elliptic curves over $\F_q$ where either the curve or its non-trivial quadratic twist has a $\F_q$-rational point of order 4. 
\end{proposition}
\begin{proof}
If $\chi(\delta ^2-\epsilon)=1$ then for some $f$ in $\F_q$ we have $\delta ^2-\epsilon=f^2$ and so, $\epsilon=(\delta -f)(\delta +f)$. If we set $d=-\delta -f$ and $a=-\delta +f$ then we can write $\J_{\epsilon,\delta }$ by  $y^2=(1-dx^2)(1-ax^2)$, where $\epsilon=ad$ and $\delta =-(a+d)/2$.
		So, Jacobi quartic curve $y^2=(1-dx^2)(1-ax^2)$ is isomorphic to  twisted Edwards Curve $\TE_{a,d}:au^2+v^2=1+du^2v^2$  under the transformations
	
		\begin{equation}
			\label{eq:TJtoJq}
			\phi(x,y)=(x,\dfrac{y}{1-dx^2})\quad ,\quad \phi^{-1}(u,v)=(u,v(1-du^2)).	
			\end{equation}

So, the subfamily of Jacobi curves $\J_{\epsilon,\delta}$ for all $\epsilon,\delta \in \F_q$ with $\chi(\delta^2-\epsilon) =1$ is the family of all elliptic curves over $\F_q$ where either the curve or its  non-trivial quadratic twist has a $\F_q$-rational point of order 4. 
\qed
\end{proof}
\section{Differential addition}
\label{sec-dadd}

A Montgomery curve $\M_{A,B}$ over a field $\F$ is given by the equation $$\M_{A,B} : By^2=x^3+Ax^2+x,$$ where $A,B$ are in $\F$ and $B(A^2-4) \ne 0$  \cite{Mo}. In Montgomery curves, the special formulas for addition and doubling is done with the $x$-coordinate of a point. Let $P_1, P_2 \in \M_{A,B}(\F)$ with $P_1\ne \pm P_2$, $2P_1\ne \cO, P_2 \ne \cO$, where $\cO$ is the point at infinity on $\M_{A,B}$. Let $x_0=x(P_2-P_1)$, $x_1=x(P_1)$, $x_2=x(P_2)$, $x_3=x(P_1+P_2)$ and $x_4=x(2P_1)$. We have 
\begin{gather*}\label{eq:Mo}
x_4	= \dfrac{(x_1^2-1)^2 }{4x_1((x_1+1)^2+(A-2)x_1)}~~~~~,~~~~x_3x_0 = \dfrac{(x_1x_2-1)^2}{(x_1-x_2)^2}~.
\end{gather*}
Since field inversion is costly, practically computations are performed where points are
represented in projective coordinates. Therefore, $x(P)$ 
is represented by $(X:Z)$ where $x(P)=X/Z$ if $P$ is an affine point and $x(P)$ is represented by $(1:0)$ if $P=\cO$. The projective differential 
addition and doubling is given as follows. 
\begin{gather*}\label{eq:MoP}
X_3=Z_0~(X_1X_2-Z_1Z_2)^2, \qquad \qquad \quad \; Z_3=X_0~(X_1Z_2-X_2Z_1)^2,\\
X_4=(X_1^2-Z_1^2)^2, \qquad Z_4=4X_1Z_1(~(X_1+Z_1)^2+(A-2)X_1Z_1~).
\end{gather*}
The cost of projective $x$-coordinate differential addition and doubling formulas for Montgomery curves $\M_{A,B}$ is $6\M+4\s+1\D$, where a multiplication in $\F$ costs one $\M$, a squaring costs one $\s$ and the cost of field multiplication by a parameter is denoted by $\D$. The $x$-coordinate of the fixed base point $P$ can be represented by $x(P)=(X : Z)$ with $Z=1$, then the differential addition and doubling formulas are computed using $5\M+4\s+1\D$. 

The Montgomery ladder performs the scalar multiplication $kP$ for a given integer $k$ and a given point $P$ on elliptic curve $E$. Here, the $x$-coordinates of the points $2mP$ and $(2m+1)P$ are computed from $x$-coordinates of the points $mP$, $(m+1)P$ and the fixed point $P$  recursively. So, in each step of Montgomery ladder both differential addition and doubling are processed that needs the same number of ground finite field operations for every bit of the scalar as a secret. Also, at the last step of the ladder $y$-coordinate of the output point $kP$ is computed.


The Montgomery method is applied to other forms of elliptic curves. Here, a suitable rational function $w$ is defined over an elliptic curve $E$.  
The function $w$ is given by fraction of polynomials in the coordinate ring of $E$ with the extra condition that $w(P)=w(-P)$ for any point on $E$. 
The $w$-coordinate \emph{differential addition} and \emph{doubling}  (\textit{dADD}) means to compute $w(2P_1)$ and $w(P_1+P_2)$ from
given values $w(P_1)$, $w(P_2)$ and $w(P_2-P_1)$, where $P_1, P_2$ are points on $E$. In practice, the computation are processed where the points are represented in the projective form. In other words, $w(P)$ is represented by $(w(P):1)$ in the projective line $\mathbb{P}(\F)$ if $w$ is regular at the point $P$.
Otherwise, it is represented by $(1:0)$.

For Montgomery curve $\M_{A,B}$, the function $w$ is defined as the $x$-coordinate of the point $P=(x,y) \in\M_{A,B}$. Now, we consider the definition of Montgomery-like formulas for any elliptic curve $E$ as follows. 

\vspace*{5pt}
\noindent \textbf{$w$-coordinate dADD.} Throughout the paper, for the function $w$ on elliptic curve $E$ over $\F$, and for the points $P_1, P_2 \in E(\F)$, let  
$w_0=w(P_2-P_1)$, $w_1=w(P_1)$, $w_2=w(P_2)$, $w_3=w(P_1+P_2)$ and $w_4=w(2P_1)$. Also, for $i=0,1,2,3,4$, $w_i$ are represented by $(W_i : Z_i)$ in the projective line $\mathbb{P}(\F).$

\begin{definition}\label{def:Mf}
Let $E$ be an elliptic curve over the field $\F$ with characteristic $p~\ne~2$. We say the rational function $w$ in the coordinate ring of $E$ over $\F$ is a Montgomery-like function on $E$ if there exist some $e\in \F$ such that for all points $P_1, P_2 \in E(\F)$ with $W_0Z_0 \ne 0$, we have 
\begin{gather*}
(~W_4~:~ Z_4~)=(~(W_1^2-Z_1^2)^2~:~ 4W_1Z_1((W_1+Z_1)^2-eW_1Z_1)~),\\
(~W_3~:~ Z_3~)=(~ Z_0~(W_1W_2-Z_1Z_2)^2~ :~ W_0~(W_1Z_2-W_2Z_1)^2~).\\
\end{gather*} 	
\end{definition}

\begin{definition}\label{def:M}
	Let $E$ be an elliptic curve over the field $\F$ with characteristic $p~\ne~2$. We say $E$ has projective differential addition and doubling Montgomery-like (\textit{dADD-M}) formulas if $E$ has some Montgomery-like function $w$ by Definition~\ref{def:Mf}. 
\end{definition}

\begin{remark}\label{rem:MI}
	Let the elliptic curve $E/\F$ has a Montgomery-like function $w$ by Definition~\ref{def:Mf} with parameter $e \in \F$. Then, for the function $\omega=1/w$ on $E$ with the notation of Definition~\ref{def:Mf}, we have 
	\begin{gather*}\label{eq:MoI}
	\omega_4	= \dfrac{4\omega_1((\omega_1+1)^2-e\, \omega_1)}{(\omega_1^2-1)^2 }~~~~~,~~~~\omega_3\omega_0 = \dfrac{(\omega_1-\omega_2)^2}{(\omega_1\omega_2-1)^2}~.
	\end{gather*}
 Note that, the projective formulas for the function $\omega=1/w$ is obtained simply by swapping the coordinates in the projective formulas of the function $w$.  
\end{remark}

\begin{definition}\label{def:MI}
	Let $E$ be an elliptic curve over the field $\F$ with characteristic $p~\ne~2$. We say the rational function $w$ in the coordinate ring of $E$ over $\F$ is an inverted Montgomery-like function on $E$ if there exist some $e\in \F$ such that for all points $P_1, P_2 \in E(\F)$ with $W_0Z_0 \ne 0$, we have 
	\begin{gather*}
	(~W_4~:~ Z_4~)=(~4W_1Z_1((W_1+Z_1)^2-eW_1Z_1) ~:~ (W_1^2-Z_1^2)^2~),\\
	(~W_3~:~ Z_3~)=(~Z_0~(W_1Z_2-W_2Z_1)^2 ~:~ W_0~(W_1W_2-Z_1Z_2)^2 ~).
	\end{gather*} 	
\end{definition}

\begin{remark}\label{rem2}
	Let $E/\F$ be an elliptic curve with inverted Montgomery-like function $\omega$ on $E$ by Definition~\ref{def:MI} with parameter $e \in \F$. From Remark~\ref{rem:MI} and Definition~\ref{def:Mf} the function $w=1/\omega$ is a Montgomery-like function on $E$.  So, $E$ has dADD-M formulas.
\end{remark}

In the next sections of this work, we show that Jacobi quartic curves over a finite field $\F_q$ with a subgroup of order 4 has dADD-M formulas with suitable inverted Montgomery-like functions. The Algorithm~\ref{dadM} gives projective  $w$-coordinate dADD-M formulas for an elliptic curve $E$ with the inverted Montgomery-like function $w$. The formulas works for all inputs except for the case where $w(P_0)$ equals $(1:0)$ or $(0:1)$. 
	
\alglanguage{pseudocode}
\begin{algorithm}
	\caption{\; Montgomery-like Projective $w$-coordinate dADD}\label{dadM}
	\begin{algorithmic}[1]
		\Statex {{\bf Input} : Elliptic curve $E$ over a field $\F$, Montgomery-like function $w$ on $E$, $e \in \F$} 
		\Statex {\bf \hspace{40pt} $(W_i: Z_i)=w(P_i),\; i=0,1,2.$} \Comment{$w(P_0)=w(P_2-P_1)$}
		\Statex {{\bf Output} : $(W_i: Z_i)=w(P_i),\; i=3,4.$} \Comment{$w(P_3)=w(P_1+P_2),\; w(P_4)=w(2P_1)$}
		\Statex {}
		\Function{\; \rm{d}ADD}{$(W_0: Z_0),\; (W_1: Z_1),\; (W_2: Z_2)$}
		\State $W_3=Z_0~(W_1Z_2-W_2Z_1)^2$
		\State $Z_3=W_0~(W_1W_2-Z_1Z_2)^2  $
		\State $W_4=4W_1Z_1(~(W_1+Z_1)^2-eW_1Z_1~) $ 
		\State $Z_4=(W_1^2-Z_1^2)^2$
		\State \Return $((W_4: Z_4),\; (W_3: Z_3))$\Comment{The differential addition and doubling }
		\EndFunction
	\end{algorithmic}
\end{algorithm}

The cost of projective $w$-coordinates differential addition and doubling formulas in Algorithm~\ref{dadM} is $6\M+4\s+1\D$. If we set $Z_0=1$, from \eqref{eq:M2}, the costs of differential 
addition and doubling formulas are $3\M+2\s $, $2\M+2\s+1\D $, respectively. And, the total cost of the mixed differential addition and doubling is 
$5\M+4\s+1\D$. 
\begin{equation}
\label{eq:M2}
\begin{array}{c}
A_1=(W_1+Z_1) ,\ B_1=(W_1-Z_1),\ A_2=(W_2+Z_2),\ B_2=(W_2-Z_2), \vspace*{3pt}\\
C=A_1B_2~ ,~\ D=A_2B_1~,~\ 	E=A_1^2-B_1^2 , \vspace*{3pt}\\
W_4	= E(A_1^2-(e/4)~E)~,~Z_4=A_1^2B_1^2, \vspace*{3pt}\\
W_3=(C-D)^2~,~Z_3=w_0(C+D)^2 
\enspace. 
\end{array}
\end{equation}

In addition, the cost of following mixed differential addition and doubling formulas \eqref{eq:M4} is $3\M+7\s+1\D$. 
\begin{equation}
\label{eq:M4}
\begin{array}{c}
A_1=(W_1+Z_1) ,\ B_1=(W_1-Z_1),\ A_2=(W_2+Z_2),\ B_2=(W_2-Z_2), \vspace*{3pt}\\
C=A_1B_2~ ,~\ D=A_2B_1~,~\ 	E=A_1^2-B_1^2 ~,~F=(A_1^4+B_1^4)-E^2, \vspace*{3pt}\\
W_4	= 2(A_1^4-(e/4)E^2)-F~~~,~~~Z_4=F, \vspace*{3pt}  \\
W_3=(C-D)^2~~~,~~~Z_3=w_0(C+D)^2 \enspace.  
\end{array}
\end{equation}

Furthermore, for the elliptic curve $E$ where $e(e-4)$ is a square in $\F$,  
the cost of following mixed differential addition and doubling formulas \eqref{eq:M6} is $3\M+6\s+3\D$. Here we obtain $r\in \F$ such that $r^2=(e-4)/e$. 
\begin{equation}
\label{eq:M6}
\begin{array}{c}
A_1=(W_1 + Z_1)~,~ B_1=(W_1 - Z_1)~,~A_2=(W_2 + Z_2)~, B_2=(W_2 -Z_2), \vspace*{3pt}\\
C= A_1~B_2~,~D= A_2~B_1~,~H_1=(rA_1^2+B_1^2)^2~,~ H_2=(rA_1^2-B_1^2)^2, \vspace*{3pt}\\
G=(H_1+H_2)~,~ K=(H_1-H_2)~,~ S=\frac{1}{r}K~ ,~ T=rK~, \vspace*{3pt}\\
W_4=2G-S-T~~,~~Z_4=T-S, \vspace*{3pt}\\
W_3=(C-D)^2~~,~~Z_3=w_0(C+D)^2 \enspace.
\end{array}
\end{equation}

From differential addition and doubling formulas \eqref{eq:M6}, the costs of differential
addition and doubling are $3\M+2\s $, $4\s+3\D $ respectively. And, the total cost of the mixed differential addition and doubling formulas is
$3\M+6\s+3\D$, where $2\D$ are the multiplication by parameter $r$ and one $\D$ is the multiplication by $1/r$. 
So, if the parameter $r$ is chosen to be small then the cost of mixed differential formulas is $3\M+6\s+1\D$.

\begin{algorithm}
	\caption{\; The Montgomery ladder}\label{MonLD}
	\begin{algorithmic}[1]
        \Statex {\bf Input : $E/\F,\; w: E(\F) \rightarrow \mathbb{P}(\F),$} \Comment{The elliptic curve $E$ over $\F$}
        \Statex {\hspace{33pt}  Projective $w$-coordinate dADD function,}
		\Statex {\bf \hspace{32pt} $k=(k_{m-1},\cdots, k_1,k_0)$} \Comment{$0\le k \in \mathbb{Z}$}
		\Statex {\bf \hspace{33pt} $(W_0:Z_0):=w(P),\, (W_1:Z_1):=w(\cO),\, (W_2:Z_2):=w(P).$}
		\Statex {\bf Output : $w(kP)$}
		\Statex {}
		\For{$i:=m-1$ {\bf down to} $0$}
        \If{$k_i=0$}
        \State {$((W_1:Z_1), (W_2:Z_2)):=dADD((W_0: Z_0),(W_1: Z_1),(W_2: Z_2))$}
        \Else
        \State {$((W_2:Z_2), (W_1:Z_1)):=dADD((W_0: Z_0),(W_2: Z_2),(W_1: Z_1))$}
        \EndIf
        \EndFor
        \State \Return $(W_1: Z_1),\; (W_2: Z_2)$ \Comment{The differential addition and doubling }
\	\end{algorithmic}
\end{algorithm} 

The Montgomery ladder is given by the Algorithm~\ref{MonLD}, that computes $w(kP)$ for any integer $k$ and any point $P$ in elliptic curve $E$ over $\F$. Here, in each step of the ladder the $w$-coordinate dADD function is required. Notice, if $E$ has dADD-~M formulas with a (inverted) Montgomery-like function $w$, then the Algorithm~\ref{MonLD}, for any integer $k$ and any point $P$, where $w(P)\ne (0:1) , (1:0)$, computes $w(kP)$ correctly. In particular, the ladder works properly even if the integer $k$ is bigger than the order of base point $P$. So, the value of $k$ can be chosen randomly as a countermeasure against differential power analysis attack.
\section{Isogeny} \label{Iso}
\begin{definition}\cite{Silv}
Let $E_1$ and $E_2$ be elliptic curves. An isogeny from $E_1$ to $E_2$ is a morphism $\phi$, satisfying $\phi(\cO_{\E_1})=\cO_{\E_2}$. 
\end{definition} 
If the kernel of a separable isogeny $\phi$ has order $l$, then $\phi$ is known as an $l$-isogeny, and $l$ is the degree of the isogeny. Also if $\phi: E_1 \rightarrow E_2$ be a nonconstant isogeny of degree $l$, then there exists a unique dual isogeny 

\vspace*{2mm}
$~~~~~~~~~~~~~~~~\hat{\phi}: E_2 \rightarrow E_1~~~~~~~~~~$ 
satisfying  
$~~~~~~~~~~~ \hat{\phi} \circ \phi=[l].$
\vspace{2mm}

 Two elliptic curves $E_1$ and $E_2$ are isogenous if there is an isogeny from $E_1$ to $E_2$ with $\phi(E_1) \neq \{\cO$\}.
Clearly if $\phi$ is an isogeny, then $\phi$ preserve the group identity. Also any nonconstant rational map $\phi$ from $E_1$ to $E_2$ which preserves the group identity is an isogeny. For a more complete reference, see \cite{Silv} and \cite{WDC8}. 
\begin{lemma}\label{ISOw}
Let $E_1$ and $E_2$ be two isogenous elliptic curve over a field $\F$ and let $\phi$ be an isogeny from $E_1$ to $E_2$. Suppose we have a differential function $\w$ on $E_2$. Then  $w=\w \circ \phi$ is a differential function on $E_1$. 
\end{lemma}
\begin{proof}
Clearly $w$ is a function on $E_1$ and for every point $P$ on $E_1$ we have $w(-P)=\w(\phi(-P))=\w(-\phi(P))=\w(\phi(P))=w(P).$ 

Since $\w$ is a differential function on $E_2$, there is a rational function $\mathcal{F}$ in $\F(x)$ that for every point $Q$ on $E_2$ we have $\w(2Q)=\mathcal{F}(\w(Q))$. 
For every point $P$ on $E_1$, $w(2P)=\w(\phi(2P))=\w(2\phi(P))=\mathcal{F}(\w(\phi(P)))=\mathcal{F}(w(P))$. 

There also exist rational function $\mathcal{G}$ on $\F(x_1,x_2,x_3)$ that for all points $Q_1, Q_2$ on $E_2$, $\w(Q_1+Q_2)=\mathcal{G}(w(Q_1),w(Q_2),w(Q_2-Q_1))$. So, for all points $P_1, P_2$ on elliptic curve $E_1$ we have, 
$w(P_1+P_2)=\w(\phi(P_1+P_2))=\w(\phi(P_1)+\phi(P_2))=\mathcal{G}(\w(\phi(P_1)),\w(\phi(P_2)),\w(\phi(P_2)-\phi(P_1)))=
\mathcal{G}(\w(\phi(P_1)),\w(\phi(P_2)),\w(\phi(P_2-P_1)))$ $=\mathcal{G}(w(P_1),w(P_2),w(P_2-P_1))$.  \qed
\end{proof}
\subsection{ Montgomery curve to weierstra\ss}
Let $\F_q$ be a field with odd characteristic and $A,B \in \F_q$ such that $B(A^2-4B)\ne 0$. 
If $\chi(B)=1$ then elliptic curve $\W_{A,B}:y^2=x(x^2+Ax+B)$, is birationally equivalent to the Montgomery curve $\M_{A/\sqrt{B},\sqrt{B}}$ under the transformation

\begin{equation*}\label{LegMon}
\phi(x,y)=(\frac{x}{\sqrt{B}},\frac{y}{B})~~~,~~~\phi^{-1}(X,Y)=(\sqrt{B}X,BY).
\end{equation*}
The Montgomery curve $\M_{A,B}$ is isogenous ( see exercise 4.11,\cite{WDC8}) to elliptic curve $\E:v^2=u(u^2 - 2ABu +B^2 (A^2 - 4)) 
$ via\\
\begin{equation*}\label{MonLeg}
\psi(x,y)=\Big(\frac{B^2y^2}{x^2},\frac{B^2y(1-x^2)}{x^2} \Big).
\end{equation*}

So, $ \varphi= \phi  \circ \psi $ is an isogeny from the Montgomery curve $\M_{A,B}$  to Montgomery curve $\M_{A',B'} $  by the map  \\ 
\begin{equation}\label{MonMon}
\varphi(x,y)=\Big(\frac{By^2}{\sqrt{A^2-4}\;x^2},\frac{y(1-x^2)}{(A^2-4)\;x^2} \Big).
\end{equation}
where $A'=\dfrac{-2A}{\sqrt{A^2-4}},\;B'=B\sqrt{A^2-4}$.

\subsection{Twisted Edwards  to Montgomery curve}
Let $\F_q$ be a field with odd characteristic and $a,d \in \F_q$ such that $ad(a-d)\ne 0$. 
A {\it twisted Edwards curve } over a field $\F_q$, is given by the equation
\begin{equation*}
	\label{eq:TE} TE_{a,d}:\quad~au^2+v^2=1+du^2v^2.
\end{equation*}

It is shown in \cite{BBLP}, that the twisted Edwards curve $\TE_{a,d}$ over a field $\F_q$ is 

birationally equivalent to the Montgomery curve $\M_{A,B}$ 
by the map 
\begin{equation}\label{TEMon}
\psi(u,v)=\Big(\frac{1+v}{1-v},\frac{1+v}{u(1-v)} \Big).
\end{equation}
where $A=2(a+d)/(a-d),\, B=4/(a-d)$. 

Also, the Montgomery curve $\M_{A,B}$ is birationally equivalent to the twisted Edwards
curve $\TE_{a,d}$ by the inverse map 
\begin{equation*}\label{TEMon2}
\psi^{-1}(x,y)=\Big(\frac{x}{y},\frac{x-1}{x+1} \Big),
\end{equation*}
where $a=(A + 2)/B,\, d = (A - 2)/B$.

Combining (\ref{MonMon}) with (\ref{TEMon}) we deduce that the twisted Edwards curve $TE_{a,d}$ and Montgomery curve $\M_{A,B}$, where $A=-(a+d)/\sqrt{ad},\, B=16\sqrt{ad}/(a-d)^2$, are isogenous 
by the map\\
\begin{gather}\label{TEMo2}
\phi(u,v)=(\dfrac{1}{\sqrt{ad}\;u^2},\dfrac{(a-d)^2v}{4ad\ u(v^2-1)}).
\end{gather}
Also the Jacobi quartic curve $J_{a^2,a-2d}$ and Montgomery curve $\M_{2-4d/a,a}$  are isogenous \cite{AG,H15}  by the map\\
\begin{equation}
\label{JM}
\psi(x,y)=\left(\dfrac{1}{ax^2},\dfrac{-y}{a^2x^3}\right)\; ~~,~~\psi^{-1}(u,v)=\left(\dfrac{1-u^2}{2av},\dfrac{a(u+1)^4+8du(u^2+1)}{4a^2v^2}\right).
\end{equation}

And the  Jacobi quartic curve $J_{a^2,a-2d}$ and twisted Edwards curve $TE_{a,d}$ are isogenous \cite{AG,H15} by the map\\
\begin{gather}\label{JTE}
\varphi(x,y)=(\dfrac{2x}{1+ax^2},\dfrac{1-ax^2}{y}) ~,~
\varphi^{-1}(u,v)=(\dfrac{u}{v},\dfrac{1-du^2v^2}{v^2}).
\end{gather}

So from combination of  isogeny maps (\ref{MonMon}) and (\ref{JM}), the Jacobi quartic curve $J_{a^2,\delta}$ and Montgomery curve $\M_{2\delta/a,a}$  are isogenous   by the map
\begin{equation}
\label{JM2}
\psi (x,y)=\left(\dfrac{y^2}{2\sqrt{\delta^2-a^2}\; x^2},\dfrac{y(1-a^2x^4)}{4(\delta^2- a^2) x^3}\right)\;.
\end{equation}

And from combination of isogenies (\ref{TEMon}) and (\ref{JM}), we can obtain the following isogeny between the Jacobi quartic curve $J_{a^2,\delta}$ and twisted Edwards curve $TE_{2(\delta+a),2(\delta-a)}$ by the map
\begin{gather}\label{JTE2}
\phi(x,y)=(\dfrac{-x}{y},\dfrac{1-ax^2}{1+ax^2}).
\end{gather}

Also from combination of  isogeny maps (\ref{TEMo2}) and (\ref{JTE}), the Jacobi quartic curve $J_{a^2,\delta}$ and Montgomery curve $\M_{A,B}$ are isogenous   by the map
\begin{equation}
\label{JMont3}
\psi (x,y)=( \dfrac{(1+ax^2)^2}{2\sqrt{2a(a-\delta)}\;x^2},\dfrac{(a+\delta)y(1-a^2x^4)}{32a(\delta-a)\;x^3}),
\end{equation}
 where $A=\dfrac{-3a+\delta}{\sqrt{2a(a-\delta)}}$ and $ B=\dfrac{32\sqrt{2a(a-\delta)}}{(a+\delta)^2}$.

\section{Differential addition on Jacobi quartic}
\label{sec-daddJ}
In this section, we extend the Montgomery method to Jacobi quartics curves $\J_{\epsilon,\delta}$ with efficient and fast $w$-coordinate differential addition and doubling. Here, $w$ is a rational function in the coordinate ring of the Jacobi curve $\J_{\epsilon,\delta}$ over $\F_q$ where $w(P)=w(-P)$ for every point $P$ in $\J_{\epsilon,\delta}(\F_q)$. 

Assume from now on that $\cP_0=P+\cT$ and $\cP_i=P+Q_i$. $(i=1,...,4)$.
\subsection{$w$-function invariant on $<\mathcal{T}>$}
We introduce two $w$-functions that  are invariant for the coset of a point up to the 2-torsion subgroup with a single point of order 2 of the $\J_{\epsilon,\delta}$ over $\F_q$.  i.e., 
$w(P)=w(Q)$ for all points $Q$ in $<\mathcal{T}>+P$.
\begin{proposition}\label{P:J1}
Let $w$ be a function on Jacobi quartic curve $\J_{\epsilon,\delta}$ over $\F_q$ given by $w(x,y)=\delta x^2$. Let $P_1$ , $P_2$ be two points of Jacobi quartic curve $\J_{\epsilon,\delta}$. Consider the $w$-coordinate dADD for $P_1$ , $P_2$. If for $i=0,1,2,3,4,$ $w_i \in \F_q$, $w_1w_2 \ne \delta^2/\epsilon  $ and $w_1^2 \ne \delta^2/\epsilon$, then we have 
\begin{gather}\label{eq:A0}
w_4	= \dfrac{4w_1(tw_1^2+1+2w_1)}{(tw_1^2-1)^2 }~~~~~,~~~~w_3w_0 = \dfrac{(w_1-w_2)^2}{(tw_1w_2-1)^2}~,
\end{gather}
where $t=\epsilon/\delta^2$.
\end{proposition}

\begin{proof}
This $w$-function  is well computed for all affine points on $\J_{\epsilon,\delta}$. For all points $P$ on the curve, we have $w(P)=w(-P)$. In particular, we have $w(\cO)=(0:1)$.

Let $P_1=(x_1,y_1)$ and $P_2=(x_2,y_2)$  be two points of Jacobi quartic curve $\J_{\epsilon,\delta}$.
From addition formula \eqref{eq:AdEJ} and  doubling formula \eqref{eq:DbEJ} we have 
\[w_4=\dfrac{4\delta x_1^2y_1^2}{(1-\epsilon x_1^4)^2}, \]

\[w_3w_0=\delta^2(\dfrac{x_2y_1+x_1y_2}  {1-\epsilon x_1^2x_2^2})^2(\dfrac{x_2y_1-x_1y_2}  {1-\epsilon x_1^2x_2^2})^2= \dfrac{(\delta x_2^2y_1^2-\delta x_1^2y_2^2)^2}  {(1-\epsilon x_1^2x_2^2)^4}.
\]
 
From  curve equation  and definition of $w$-function we have  $y_i^2=tw_i^2+2w_i+1$.
Now by  straightforward calculation, we obtain 
\begin{equation*}
w_4	= \dfrac{4w_1(tw_1^2+2w_1+1)}{(tw_1^2-1)^2 }
\end{equation*}
and
 \begin{equation*}
  w_3w_0=\dfrac{(w_1(tw_2^2+2w_2+1)-w_2(tw_1^2+2w_1+1))^2}  {(1-tw_1w_2)^4}= \dfrac{(w_1-w_2)^2}{(tw_1w_2-1)^2}. 
 \end{equation*}
 \end{proof} \qed
%
%


Assume that $w_0$ is given as a field element, and the inputs $w_1, w_2$ are given as fractions
$W_1/Z_1$, $W_2/Z_2$ and the outputs $w_4, w_3 $ are given as fraction $W_4/Z_4$ and  $W_3/Z_3$. From
Eq.~\eqref{eq:A0} the explicit projective formulas are given by
\begin{equation}
\label{eq:P00}
\begin{array}{c}
\dfrac{W_4}{Z_4}= \dfrac{4W_1Z_1(~tW_1^2+Z_1^2+2W_1Z_1)} {(tW_1^2-Z_1^2)^2}, \vspace*{7pt}\\
\dfrac{W_3}{Z_3}= \dfrac{Z_0~(W_1Z_2-W_2Z_1)^2}  {W_0~(tW_1W_2 -Z_1Z_2)^2}\enspace. 
\end{array}
\end{equation}

The cost of this projective $w$-coordinates addition and doubling formulas is $6\M+6\s+2\D$.
If we set $Z_0=1$, then we obtain the following mixed projective $w$-coordinates differential addition and doubling formulas with cost of $5\M+6\s+2\D$.
\begin{equation}
\label{eq:M0}
\begin{array}{c}
A_1=(W_1+Z_1) ,\  A_2=(W_2-Z_2), B_1=W_1W_2,\ B_2=Z_1Z_2 \vspace*{3pt}\\
C=A_1A_2~ ,~\ D=B_1-B_2~,~\ 	E=A_1^2-(W_1^2+Z_1^2) , \vspace*{3pt}\\
W_4	= 2E(tW_1^2+Z_1^2+E)~,~Z_4=(tW_1^2-Z_1^2)^2, \vspace*{3pt}\\
W_3=(C-D)^2~,~Z_3=w_0(tB_1-B_2)^2 
\enspace. 
\end{array}
\end{equation}
\begin{proposition}\label{P:J01}
Fix a field  $\F_q$ with odd charactristic and $\epsilon,\delta \in \F_q $ such that $\delta(\delta^2-\epsilon) \ne 0$. Define $\J_{\epsilon,\delta}$ as the Jacobi quartic curve $y^2=\epsilon x^4+2\delta x^2+1$. Define $w:~J_{\epsilon,\delta}(\F_q) \mapsto \F_q  $ as fallow : $w(x,y)=\delta x^2$.

Let $W_1,Z_1,W_2,Z_2,W_3,Z_3,W_4,Z_4$ be elements of $\F_q$ and $t=\epsilon/\delta^2$. Define
\begin{equation}
\label{eq:PJ0}
\begin{array}{c}
W_4= 4W_1Z_1(~tW_1^2+Z_1^2+2W_1Z_1), \vspace*{7pt}\\
Z_4= (tW_1^2-Z_1^2)^2,~~~~~~~~~~~~~~~~~~~~~~\vspace*{7pt}\\
W_3= Z_0~(W_1Z_2-W_2Z_1)^2,~~~~~~~~~~~~\vspace*{7pt}\\
Z_3= W_0~(tW_1W_2 -Z_1Z_2)^2.~~~~~~~~~~~\enspace 
\end{array}
\end{equation}

\item[(i)] Let $P_1$ be an element of $\J_{\epsilon,\delta}$. If $W_1 \ne 0 $ or $ Z_1 \ne 0 ;~w(P_1)=W_1/Z_1$, then $W_4/Z_4 \ne 0/0$ and $w(2P_1)=W_4/Z_4$. 


\item[(ii)] Assume that $W_0 \ne 0; ~Z_0 \ne 0;~w(P_0)=W_0/Z_0. ~W_1 \ne 0 $ or $ Z_1 \ne 0 ;~w(P_1)=W_1/Z_1 .$ 
 $W_2 \ne 0$ or $ ~Z_2 \ne 0;~w(P_2)=W_2/Z_2$,  then $W_3/Z_3 \ne 0/0$ and $w(P_1+P_2)=W_3/Z_3$. 

$~$

Here $W/Z$ means the quotient of $W$ and $Z$ in $\F_q$  if $Z \ne 0$; it means $\infty$ if $W \ne 0$ and $Z=0$; it is undefined if $W=Z=0$.
\end{proposition}
\begin{proof}

\item[(i)]
We have two casese $W_1 Z_1=0 $ and $W_1Z_1 \neq 0$.
\begin{itemize}
\item[$\bullet$] If $W_1 Z_1=0 $ then $Z_1=0, W_1 \neq 0$ or $W_1=0, Z_1 \neq 0$.
\begin{description}
\item[-] If $Z_1=0$ then $Z_4=t^2W_1^4\ne 0 $ and $W_4=0$ so $(W_4:Z_4) = (0:1)$ and  $W_4/Z_4 \ne 0/0$.
Since  $Z_1=0$ so $P_1= \infty$ or $ \overline{\infty}$ and $2P_1=\cO$. Hence $w(2P_1)=(0:1)$ as claimed.
\item[-] If $W_1=0$ then $Z_4=Z_1^4\ne 0 $ and $W_4=0$ so $(W_4:Z_4) = (0:1)$ and  $W_4/Z_4 \ne 0/0$. 
From  $W_1=0, Z_1 \neq 0$ we have $2P_1= \cO $, so $w(2P_1)=w(\cO)=(0:1)$ as claimed.
\end{description} 

\item[$\bullet$] If $W_1Z_1 \ne 0$ then $P_1 \ne \infty , \overline{\infty}$. We show that $W_4/Z_4\neq 0/0$. 

Assume that $W_4/Z_4=0/0$ then 
\begin{equation*}
\begin{cases}
tW_1^2+Z_1^2+2W_1Z_1=0, \quad ~\\
\quad ~\\
tW_1^2=Z_1^2 .
\end{cases}
\end{equation*}
By solving this system, we have 
$$2Z_1^2+2W_1Z_1=0\Rightarrow 2Z_1(Z_1+W_1)=0 \Rightarrow W_1=-Z_1.$$
So $t(-Z_1)^2=Z_1^2$ so $ t=\delta^2/\epsilon=1$, contradicting the hypothesis that $\delta^2-\epsilon \ne 0$.  So  $W_4/Z_4 \ne 0/0$ as claimed. Also from  $W_1~Z_1 \ne 0$ we  know that, the point of $P_1$ is affine point, so from Proposition \eqref{P:J1}  we have  $w(2P_1)=W_4/Z_4$. 
\end{itemize}
 $~$
 
\item[(ii)] Assume that $W_0Z_0 \neq 0$. 
\begin{itemize}
\item If $P_2=P_1$ then $P_0=\cO$ so $W_0=0$, contradiction. Hence $P_1 \ne P_2$.

\item If $P_1= \infty,~P_2= \overline{\infty}$ then $Z_1=Z_2=0$ and $W_1W_2 \neq 0$, so $Z_3=W_0(tW_1W_2)^2 \ne 0$ and $W_3=0$, so $(W_3:Z_3) = (0:1)$. Hence  $W_3/Z_3 \ne 0/0$. Also from addition formula \eqref{eq:TEH1} we have $P_1+P_2=\mathcal{T}$, so $w(P_1+P_2)=w(\mathcal{T})=(0:1)$ as claimed.

\item If $P_1= \infty$ or $P_1= \overline{\infty}$ and  $~P_2 \ne \infty, \overline{\infty}$ then $Z_1=0$ , $Z_2 \ne 0$. We know that $w(\infty)=w(\overline{\infty})=(1:0)$, 
so $W_1=1$. Hence $W_3=Z_0Z_2^2$ and $Z_3=t^2W_0W_2^2$. Since $Z_0 Z_2 \ne  0$  we have $W_3 \ne 0$. Hence   $W_3/Z_3 \ne 0/0$ as claimed. Now, if $Z_3 =0$ then $P_1+P_2=\infty$ or $P_1+P_2=\overline{\infty}$, so $w(P_1+P_2)=(1:0)$. Also,  if $Z_3 \ne 0$ then $P_1+P_2$ is affine point, so from Proposition \eqref{P:J1}  we have  $w(P_1+P_2)=W_3/Z_3$. 
\end{itemize}
\noindent Assume from now on that  $P_1$ and $P_2$ be affine points and $P_1 \ne  P_2$.  
\begin{itemize}
\item If $P_2=-P_1$ then $W_1/Z_1=w(P_1)=w(-P_2)=w(P_2)=W_2/Z_2$, so $W_1Z_2=W_2Z_1$, so $W_3=0$.
Note that $W_1 W_2 \ne 0:$ if  $W_1 = 0$ then  $Z_1 \ne 0$, so  $w(P_1) =0$, so $P_1=\cO$ or $P_1=\mathcal{T}$, so $P_2=-P_1=P_1$, contradiction. 
Similar arguments apply to the case $W_2= 0$.

 We will show that $Z_3 \ne 0$, hence $(W_3:Z_3) = (0:1)$ and  $W_3/Z_3 \ne 0/0$ as claimed.

Assume that  $Z_3=0$, then $W_0(tW_1W_2-Z_1Z_2)=0$. But $W_0 \ne 0$, so $(tW_1W_2-Z_1Z_2)=0$ so $\frac{tW_2}{Z_2}=\frac{Z_1}{W_1}$.  In addition, we have $\frac{W_2}{Z_2}=\frac{W_1}{Z_1}$. Hence $\frac{tW_1}{Z_1}=\frac{Z_1}{W_1}$, so $tW_1^2-Z_1^2=0$. So from doubling formula we have $Z_4=0$, But from $P_2=-P_1$ we hve  $P_0=2P_1$, so  $Z_0=Z_4=0$, contradiction. Also we have $P_1+P_2=\cO$, so $w(P_3)=w(\cO)=(0:1)$ as calimed.

\item  $P_2 \ne \pm P_1$. 

If $W_3 = 0$, then $Z_0(W_1Z_2-W_2Z_1)=0$, but $Z_0 \ne 0$, so $W_1Z_2-W_2Z_1=0$. We know that, $P_1$ and $P_2$ are affine points, so $Z_1 \ne 0$ and $Z_2 \ne 0$, hence $W_1/Z_1=W_2/Z_2$, so $w(P_1)=w(P_2)$, hence $P_2=\pm P_1$,  contradiction. So $W_3 \ne 0$ and $W_3/Z_3 \ne 0$ as claimed. 
On the other hand, the points of $P_1$ and $P_2$ are affine points, so from proposition \eqref{P:J1}  we have  $w(P_1+P_2)=W_3/Z_3$. 
\end{itemize}
 \qed
\end{proof}

Now, we consider the family of all elliptic curves over $\F_q$ with full 2-torsion points, i.e. the family of Jacobi curves $\J_{\epsilon,\delta}$ with $\chi(\epsilon) =1$, and define the rational function $w$ by $w(x,y)=\sqrt{\epsilon}~x^2$. 
\begin{proposition}\label{P:J2}
		The Jacobi quartic curve $\J_{\epsilon,\delta}$ over $\F_q$ of odd characteristic with $\chi(\epsilon)=1$ has dADD-M formulas with the function $w$ by $w(x,y)=\sqrt{\epsilon} x^2$ and $e={2(\sqrt{\epsilon}-\delta )}/{\sqrt{\epsilon}}$.
		%
		\end{proposition}

\begin{proof}
The Jacobi quartic curve  $J_{\epsilon,\delta}$ with $\chi(\epsilon)=1$ and Montgomery curve $\M_{2\delta/\sqrt{\epsilon}\, ,\sqrt{\epsilon}}$ are isogenous \cite{AG,H15}  by the map
	
	\begin{equation*}
			\psi (x,y)=(\dfrac{1}{\sqrt{\epsilon}\, x^2},\dfrac{-y}{\epsilon\, x^3}).
		\end{equation*}
We know that Montgomery curve $\M_{A,B}$ has dADD-M formulas with the function $\w$ by $\w(X,Y)=X$ and $e=2-A$. So from Lemma \ref{ISOw}, $\omega=\w\circ \psi$ is a $w$-function. So we have 
$\omega(x,y)=\w( \psi(x,y))=\frac{1}{\sqrt{\epsilon}\, x^2}$. So from remak \ref{rem2},  Jacobi quartic curve $J_{\epsilon,\delta}$ has dADD-M formulas with the function $w=1 /\omega $. So $w(x,y)=\sqrt{\epsilon} x^2$ and $e=2-A=2-2\delta/\sqrt{\epsilon}=2(\sqrt{\epsilon}-\delta)/\sqrt{\epsilon}$. \qed
	\end {proof}
Similarly, we have the same mixed differential addition and doubling formulas as  \eqref{eq:M2} ,\eqref{eq:M4} with cost of $5\M+4\s+1\D$ and $3\M+7\s+1\D$, respectively.  
Also, for the Jacobi quartics curves $\J_{\epsilon,\delta}$ with $\chi(\epsilon)= \chi(\delta^2-\epsilon)=1$, by assuming
$r=\dfrac{\delta+\sqrt{\epsilon}}{\sqrt{\delta^2-\epsilon}}$
, we have the same mixed projective differential addition and doubling formulas as \eqref{eq:M6} with cost of $3\M+6\s+3\D$.
Notice, $\delta=\frac{r^2+1}{r^2-1}\sqrt{\epsilon}$. So, if the parameter $r$ is chosen to be small then the cost of mixed differential formulas is $3\M+6\s+1\D$.

\begin{example}
Let $p=2^{221}-3$ and $w(x,y)=\sqrt{\epsilon}\,x^2$. If $\epsilon=1$ and $\delta= 16254858882142650\\
\qquad 59385149248953355680048842182580119166027190412745$, then $\chi(\delta^2-\epsilon)=1$  and the  Jacobi quartic curve $\J_{\epsilon,\delta}$ over  $\F_p$ is of order $4\ell$, where $\ell$ is the prime 
\begin{gather*} 
\ell=842498333348457493583344221469363\\
\qquad 786243439662766001461011538653297.
\end{gather*}
The cost of the mixed differential addition and doubling formulas \eqref{eq:M6} is
$3\M+6\s+3\D$, where $2\D$ is the multiplication by the small constant $r=3096$ and one $\D$ is the multiplication by $1/r$.
\end{example}


\begin{example}
Let $p=2^{255}-19$ and $w(x,y)=\sqrt{\epsilon}\,x^2$. If $\epsilon=1$ and $\delta= 
539925905847637\\
\qquad 56419333460221746051882217470897926824441922037652483784944199$, then $\chi(\delta^2-\epsilon)=1$  and the  Jacobi quartic curve $\J_{\epsilon,\delta}$ over  $\F_p$ is of order $4\ell$, where $\ell$ is the prime 
\begin{gather*} 
\ell=144740111546645244279463731260859\\
\qquad 88481696135633646439402133744090170773571421.
\end{gather*}
The cost of the mixed differential addition and doubling formulas \eqref{eq:M6} is
$3\M+6\s+3\D$, where $2\D$ is the multiplication by the small constant $r=62094$ and one $\D$ is the multiplication by $1/r$.
\end{example}

\begin{table}[h!]
		\caption{}\label{tab2}
	\renewcommand{\arraystretch}{2}
	\begin{tabular}{l|l|l|l|l|l|l}
		\hline
	$w(P)$&$w(\cP_0)$ & $w(\cP_1)$ & $w(\cP_2)$ & $w(\cP_3)$ & $w(\cP_4)$ & \text{invariant on subgroup}\\
		\hline
	$\varpi_1=\sqrt{\epsilon}x^2 $& $\frac{1}{\varpi_1}$&$\varpi_1' $&$\frac{1}{\varpi_1'}$&$\varpi_1'' $&$\varpi_1'''$ & $<\mathcal{T}> \cong \Z_2$ \\
	\hline
\multicolumn{7}{c}{\tiny where
$ \varpi_1'=\dfrac{\sqrt{\epsilon}r_1y^2}{(1-\epsilon r_1x^2)^2},
\varpi_1''=  \dfrac{\sqrt{\beta_1}\,x+\sqrt{\alpha_1}\,y}{\sqrt{\alpha_1}\,y-\sqrt{\beta_1}\,x},
  \varpi_1'''= \dfrac{\sqrt{\beta_2}\,x+\sqrt{\alpha_2}\,y}{\sqrt{\beta_2}\,x-\sqrt{\alpha_2}\,y}.$
}
 \end{tabular}
\end{table}
\paragraph{{\bf Retrieving $x$ and $y$ from $w$-Coordinates.}}
Assume taht $P=(x_1,y_1),Q=(x_2,y_2),Q-P=(x_0,y_0), w(P)=w_1,w(Q)=w_2$ and $w(Q-P)=w_0$. 
From doubling formula \ref{eq:DbEJ} and $w$-function $w(x,y)=\sqrt{\epsilon}\; x^2$,  we have 
 \begin{equation}\label{eq:2kP}
2(x_1,y_1)=(\dfrac{2x_1y_1}{1-w_1^2},
\dfrac{(w_1^2+\frac{4\delta}{\sqrt{\epsilon}} +1)(w_1^2+1)+4w_1^2}{(1-w_1^2)^2} ).
\end{equation}
Also, if $x_0y_0\ne 0$ then
\begin{equation}\label{eq:Recovery}
x_1y_1=\dfrac{  w_2(1-w_0w_1)^2 -(w_0+w_1)(1+w_0w_1)- \frac{4\delta}{\sqrt{\epsilon}}w_0w_1 }{2\sqrt{\epsilon}\; x_0y_0}.
\end{equation}

So, one can use this formula to recover $2P$ given $Q-P,w(P),w(Q)$.
In particular, we can recover $2kP$ given $P,w(kP),w((k+1)P)$. So the Montgomery ladder can be used not only to compute $w(kP)$ given $w(P)$, but also to compute $2kP$ given $P$. Algorithm \ref{J:Recovery} describes the Montgomery ladder approach for point multiplication of $2kP$.
%
%
\begin{algorithm}
\caption{\; Recovering $2kP$ from $w$-coordinates dADD}\label{J:Recovery}
\begin{algorithmic}[1]
 \Statex {\bf Inputs :}
 {Elliptic curve  $ \J_{\epsilon,\delta}/F_{q}$,~ Point $P=(x,y) $ on $ \J_{\epsilon,\delta}$ } and integer k
 \Statex {\bf Output : $Q=2kP$} 
 \Statex {}
 \State  set: $w_0=\sqrt{\epsilon}\; x^2$
   \State Compute: $w_1\leftarrow w(kP), w_2\leftarrow w((k+1)P) $ \Comment{dADD }
   \State Compute: $x_1y_1$ from Eq.\ref{eq:Recovery}  \Comment{$kP=(x_1,y_1)$ }
   \State Compute $2kP$ from Eq. \ref{eq:2kP}  
   \State \Return $Q=2kP$

\end{algorithmic}
\end{algorithm}

We note that, in typical cryptography application, $P$ has odd order $\mathit{l}$. So one can replace $k$ by $k/2~ mod ~\mathit{l}$, obtaining $kP=2\left( (k/2)~mod ~l \right)P $ when given $P$ and $w((k/2)~mod ~l )P)$.

Let $k=(k_{m-1},\cdots, k_1,k_0)$ and $k'=(k_{m-1},\cdots, k_2,k_1)$, then  $k=2k'+k_0$. So  if $k$ be an even integer, $kP=2k'P$ and if $k$ be an odd integer, $kP=2k'P+P$. So, scaler multiplication can be modified  as in Algorithm \ref{J:Recovery2} for computing $kP$ when given $P,w(k'P)$ and $w((k'+1)P)$.

\begin{algorithm}
\caption{\; Retrieving  $kP$ from $w$-coordinates dADD}\label{J:Recovery2}
\begin{algorithmic}[1]
 \Statex {\bf Inputs :} \Statex {\hspace{32pt}  Elliptic curve  $ \J_{\epsilon,\delta}$ }
\Statex {\hspace{32pt} $P=(x,y) \in  \J_{\epsilon,\delta}$   } 
 \Statex {\hspace{32pt} $ k=(k_{m-1},\cdots, k_1,k_0)_2$} \Comment{$k$ is a positive integer}

 \Statex {\bf Output : $Q=kP$} 
 \Statex {}
 \State set: $k'=(k_{m-1},\cdots, k_1)$
 \State  set: $w_0=\sqrt{\epsilon}\; x^2$  
   \State Compute: $w_1\leftarrow w(k'P), w_2\leftarrow w((k'+1)P) $ \Comment{dADD }
   \State Compute: $x_1y_1$ from Eq.\ref{eq:Recovery}  \Comment{$k'P=(x_1,y_1)$ }
   \State Compute $2k'P$ from doubling formula \ref{eq:2kP}  
  \State set: $Q'=2k'P $  
  \State set: {$Q=Q'+k_0P$}\Comment{$k=2k'+k_0$ }
   \State \Return $Q$
\end{algorithmic}
\end{algorithm}

\subsection{$w$-function invariant on $<\mathcal{T},\infty>$}
Now, we consider another $w$-functions that are invariant for the coset of a point up to the full 2-torsion subgroup of the Jacobi quartic curve  $J_{\epsilon,\delta}$ over $\F_q$, i.e. 
$w(P)=w(Q)$ for all points $Q$ in $\J[2]+P$. Note that this $w$-functions could be applicable with the eliminating cofactors technique through point compression \cite{H15}.
\begin{proposition}\label{P:J3}
Let $w$ be a function on Jacobi quartic curve $\J_{\epsilon,\delta}$ over $\F_q$ given by $w(x,y)=2\delta x^2/y^2$. Let $P_1$ , $P_2$ be two points of Jacobi quartic curve $\J_{\epsilon,\delta}$. Consider the $w$-coordinate dADD for $P_1$ , $P_2$. If for $i=0,1,2,3,4,$ $w_i \in \F_q$, $w_1w_2 \ne 1/(\delta^2-\epsilon)  $ and $w_1^2 \ne 1/(a^2-d)$, then we have
\begin{gather}\label{eq:A3}
w_4	= \dfrac{4w_1((w_1-1)^2-\epsilon / \delta^2 w_1)}{((1-\epsilon / \delta^2) w_1^2-1)^2 }~~,~~w_3w_0 = \dfrac{(w_1-w_2)^2}{((1-\epsilon / \delta^2)w_1w_2-1)^2}~.
\end{gather}
\end{proposition}
\begin{proof}
Assume that $P_1=(x_1,y_1)$ and $P_2=(x_2,y_2)$  be two points of Jacobi quartic curve $\J_{\epsilon,\delta}$ and $w(x,y)=2\delta x^2/y^2$. Using doubling formula \eqref{eq:DbEJ} we have
\begin{equation} \label{eq:AP1}
w_4=2\delta \left(\dfrac{x(2P_1)}{y(2P_1)}\right)^2=\dfrac{8\delta x_1^2y_1^2{(1-\epsilon x_1^4)^2}}{(~(y_1^2+2\delta x_1^2)(1+\epsilon x_1^4)+4\epsilon x_1^4~)^2}  
\end{equation}
But from curve equation \eqref{eq:EJacob}, we have $1+\epsilon x_1^4=y_1^2-2\delta x_1^2$. On the other hand
\begin{equation*}
(1-\epsilon x_1^4)^2=(1+\epsilon x_1^4)^2-4\epsilon x_1^4=(y_1^2-2\delta x_1^2 )^2-4\epsilon x_1^4.
\end{equation*}
Hence
\begin{equation*}
(1-\epsilon x_1^4)^2=y_1^4+4(\delta^2-\epsilon)x_1^4-4\delta x_1^2y_1^2.
\end{equation*}
From formula \eqref{eq:AP1}, we have
\begin{equation*}
w_4=\dfrac{8\delta x_1^2/y_1^2 \left(1+(\delta^2-\epsilon)(2x_1^2/y_1^2) ^2-4\delta x_1^2/y_1^2\right)}{\left(1-(\delta^2-\epsilon)(2x_1^2/y_1^2) ^2\right)^2}  
\end{equation*}
Since $w_1=2\delta x_1^2/y_1^2$, we obtain   
\begin{gather}\label{eq:AP2}
w_4	= \dfrac{4w_1((w_1-1)^2-\epsilon / \delta^2 w_1^2)}{((1-\epsilon / \delta^2) w_1^2-1)^2 }.
\end{gather}
Now, let $w_3=w(P_1+P_2)$ and $w_0=w(P_2-P_1)$. 
Using the addition formula \eqref{eq:AdEJ} we have
  \begin{equation*}
 w_3=2\delta \left(~\dfrac{( x_2y_1+x_1y_2) (1-\epsilon x_1^2x_2^2) }{(y_1y_2+2\delta x_1x_2)(1+\epsilon x_1^2x_2^2)+2\epsilon x_1x_2(x_1^2x_2^2)} \right)^2,
 \end{equation*}
and
  \begin{equation*}
 w_0=2\delta \left(~\dfrac{( x_2y_1-x_1y_2) (1-\epsilon x_1^2x_2^2) }{(y_1y_2-2\delta x_1x_2)(1+\epsilon x_1^2x_2^2)-2\epsilon x_1x_2(x_1^2x_2^2)} \right)^2.
 \end{equation*}
Assume that
\begin{equation*}
\begin{array}{c}
   \quad A= (y_1y_2+2\delta x_1x_2)(1+\epsilon x_1^2x_2^2)+2\epsilon x_1x_2(x_1^2x_2^2)  ~,\vspace*{3pt}\\
B=(y_1y_2-2\delta x_1x_2)(1+\epsilon x_1^2x_2^2)-2\epsilon x_1x_2(x_1^2x_2^2)~.
\end{array}
\end{equation*}
%
Using curve equation \eqref{eq:EJacob}, we have   $AB=(1-\epsilon x_1^2x_2^2)^2(y_1^2y_2^2-4(\delta^2-\epsilon)x_1^2x_2^2)$.
Therefore
 \begin{equation*}
 w_3w_0=\left(\dfrac{2\delta ( x_2^2y_1^2-x_1^2y_2^2) (1-\epsilon x_1^2x_2^2)^2 }{AB} \right)^2=\left(\dfrac{2\delta ( x_2^2y_1^2-x_1^2y_2^2) (1-\epsilon x_1^2x_2^2)^2 }{(1-\epsilon x_1^2x_2^2)^2(y_1^2y_2^2-4(\delta^2-\epsilon)x_1^2x_2^2)}\right)^2
 \end{equation*} 
\begin{equation*}
 =\left(~\dfrac{2\delta ( x_2^2y_1^2-x_1^2y_2^2)  }{(y_1^2y_2^2-4(\delta^2-\epsilon)x_1^2x_2^2)} \right)^2.
 \end{equation*}
By the substitution $w_i=2\delta x_i^2/y_i^2$ for $i=1,2$, we have
\begin{equation*} \label{AP3}
 w_3w_0=\dfrac{(w_1-w_2)^2}{((1-\epsilon / \delta^2)w_1w_2-1)^2}.
\end{equation*}   \qed
\end{proof}
As before assume that $w_0$ is given as a field element, and the inputs $w_1, w_2$ are given as fractions
$W_1/Z_1$ , $W_2/Z_2$ and the outputs $w_4, w_3 $ are given as fraction $W_4/Z_4$ and  $W_3/Z_3$. From
Eq.~\eqref{eq:A3} the explicit projective formulas are given by
\begin{equation}
\label{eq:P01}
\begin{array}{c}
\dfrac{W_4}{Z_4}= \dfrac{4W_1Z_1(~(W_1-Z_1)^2-\epsilon / \delta^2 \ W_1^2)} {((1-\epsilon / \delta^2)W_1^2-Z_1^2)^2}, \vspace*{7pt}\\
\dfrac{W_3}{Z_3}= \dfrac{Z_0~(W_1Z_2-W_2Z_1)^2}  {W_0~((1-\epsilon / \delta^2)W_1W_2 -Z_1Z_2)^2}\enspace. 
\end{array}
\end{equation}
If we set $Z_0=1$, then the following mixed projective $w$-coordinates differential addition and doubling formulas
have the total cost $5\M+6\s+2\D $ as follows:
\begin{equation}
\label{eq:M8}
\begin{array}{c}
A_1=(W_1-Z_1) ,\  A_2=(W_2+Z_2), B_1=W_1W_2,\ B_2=Z_1Z_2 \vspace*{3pt}\\
C=A_1A_2~ ,~\ D=B_1-B_2~,~\ 	E=W_1^2+Z_1^2-A_1^2 , \vspace*{3pt}\\
W_4	= 2E(A_1^2-(\epsilon / \delta^2) \ W_1^2)~,~Z_4=(W_1^2-Z_1^2-(\epsilon / \delta^2) \ W_1^2)^2, \vspace*{3pt}\\
W_3=(C-D)^2~,~Z_3=w_0((1-\epsilon / \delta^2) \ B_1-B_2)^2 
\enspace. 
\end{array}
\end{equation}

\begin{proposition}\label{P:J4}
The Jacobi quartic curve $\J_{\epsilon,\delta}$ over the finite field $\F_q$ of odd characteristic with $\chi(\delta^2-\epsilon)=1$ has dADD-M formulas with the function $w$ given by $$w(x,y)=2\sqrt{\delta^2-\epsilon}\;\dfrac{x^2}{y^2}$$ and $e={2(\delta+\sqrt{\delta^2-\epsilon})}/{\sqrt{\delta^2-\epsilon}}$.
\end{proposition}
\begin{proof}
The proof is similar to the proof of Proposition \eqref{P:J3}. \qed
\end{proof}

Again, we have the same projective and mixed differential formulas as \eqref{eq:M2} ,\eqref{eq:M4}. 
on the other hand, for the Jacobi quartics curves $\J_{\epsilon,\delta}$ with $\chi(\epsilon)= \chi(\delta^2-\epsilon)=1$, by assuming 
$r=\frac{\sqrt{\epsilon}}{\delta+\sqrt{\delta^2-\epsilon}}$, we have the same mixed differential addition and doubling formulas as \eqref{eq:M6}.
\begin{proposition}\label{P:J7}
The Jacobi quartic curve $\J_{\epsilon,\delta}$ over the finite field $\F_q$ of odd characteristic with  $\chi(2\sqrt{\epsilon}(\sqrt{\epsilon}-\delta))=1$ has dADD-M formulas with the function $w$ given by $$w(x,y)=2\sqrt{2\sqrt{\epsilon}(\sqrt{\epsilon}-\delta)}\;(\dfrac{x}{1+\sqrt{\epsilon} x^2})^2$$ and
 $e=(2\sqrt{\epsilon}+\sqrt{2\sqrt{\epsilon}(\sqrt{\epsilon}-\delta)}~)^2/(2\sqrt{\epsilon}\sqrt{2\sqrt{\epsilon}(\sqrt{\epsilon}-\delta)})$.
\end{proposition}
\begin{proof}
From isogeny map \eqref{JMont3} the Jacobi quartic curve $J_{\epsilon,\delta}$ with $\chi(\epsilon)=1$ and Montgomery curve $\M_{A,B}$
 are isogenous by the map
\begin{gather*}
\psi (x,y)=( \dfrac{(1+\sqrt{\epsilon}\,x^2)^2}{2\sqrt{2\sqrt{\epsilon}(\sqrt{\epsilon}-\delta)}\;x^2},\dfrac{(\sqrt{\epsilon}+\delta)(1-\epsilon x^4)y}{32\sqrt{\epsilon}(\delta - \sqrt{\epsilon})\;x^3}).
\end{gather*}

But the Montgomery curve $\M_{A,B}$ has dADD-M formulas with the function $v$ by  $v(X,Y)=X$ and $e=2-A$. 
So from Lemma \eqref{ISOw}, $\omega=v\circ \psi$ is a $w$-function. So we have 
$\omega(x,y)=v( \psi(x,y))=\dfrac{(1+\sqrt{\epsilon}\,x^2)^2}{2\sqrt{2\sqrt{\epsilon}(\sqrt{\epsilon}-\delta)}\;x^2}$. Then, for function $\omega=1/w$ on Jacobi quartic curve $J_{\epsilon,\delta}$ we have $\omega(x,y)=v( \psi(x,y))=2\sqrt{2\sqrt{\epsilon}(\sqrt{\epsilon}-\delta)}\;(\dfrac{x}{1+\sqrt{\epsilon}\,x^2})^2$
and  $e=2-\dfrac{-3\sqrt{\epsilon}+\delta}{\sqrt{2\sqrt{\epsilon}(\sqrt{\epsilon}-\delta)}}=(2\sqrt{\epsilon}+\sqrt{2\sqrt{\epsilon}(\sqrt{\epsilon}-\delta)}~)^2/(2\sqrt{\epsilon}\sqrt{2\sqrt{\epsilon}(\sqrt{\epsilon}-\delta)})$. \qed
\end {proof}
Similarly, we have the same mixed differential addition and doubling formulas as  \eqref{eq:M2} ,\eqref{eq:M4}. 
Also, for the Jacobi quartics curves $\J_{a^2,\delta}$ with $\chi(2\sqrt{\epsilon}(\sqrt{\epsilon}-\delta))=1$, by assuming
$r=\dfrac{2\sqrt{\epsilon}-\sqrt{2\sqrt{\epsilon}(\sqrt{\epsilon}-\delta)}}{2\sqrt{\epsilon}+\sqrt{2\sqrt{\epsilon}(\sqrt{\epsilon}-\delta)}}$
, we have the same mixed projective differential addition and doubling formulas as \eqref{eq:M6} with cost of $3\M+6\s+3\D$. So, if the parameter $r$ is chosen to be small then the cost of mixed differential formulas is $3\M+6\s+1\D$. 

\begin{example}
Let $p=2^{221}-3$ and $w(x,y)=2\sqrt{2\sqrt{\epsilon}(\sqrt{\epsilon}-\delta)}\, \left(\dfrac{x}{1+\sqrt{\epsilon}\,x^2}\right)^2$. If $\epsilon=1$, $\delta=
3208532527536442342869090771709507549919518828153330010014177629343$ then the  Jacobi quartic curve $\J_{\epsilon,\delta}$ over  $\F_p$ is of order $8\ell$, where $\ell$ is the prime 
\begin{gather*} 
\ell=421249166674228746791672110734681\\
\qquad 565429440931821392159528718493639.
\end{gather*}
The cost of the mixed differential addition and doubling formulas \eqref{eq:M6} is
$3\M+6\s+3\D$, where $2\D$ is the multiplication by the small constant $r=3096$ and one $\D$ is the multiplication by $1/r$.
\end{example}
\begin{tiny}{
\begin{table}[h!]
		\caption{}\label{tab3}
	\renewcommand{\arraystretch}{2}
	\begin{tabular}{l|l|l|l|l|l|l}
		\hline
	$w(P)$&$w(\cP_0)$ & $w(\cP_1)$ & $w(\cP_2)$ & $w(\cP_3)$ & $w(\cP_4)$ & \text{invariant on subgroup}\\
		\hline
		$\varpi_2=2\sqrt{\delta ^2-\epsilon}~\frac{x^2}{y^2} $ &$\varpi_2$&$\frac{1}{\varpi_2}$&$\frac{1}{\varpi_2}$&$\varpi_2'$ &$\frac{1}{\varpi_2'}$ &$<\mathcal{T},\infty> \cong  \Z_2 \times \Z_2$	\\
		\hline
		$\varpi_3=\frac{2\sqrt{2(\epsilon -\sqrt{\epsilon} \delta)}\;x^2}{(1+\sqrt{\epsilon} x^2)^2} $&$\varpi_3 $&$\varpi_3'$ &$\varpi_3'$&$\frac{1}{\varpi_3'}$&$\frac{1}{\varpi_3} $&$ <\mathcal{T},\infty> \cong  \Z_2 \times \Z_2$\\
		\hline
\multicolumn{7}{c}{\tiny where 
$\varpi_2'=\dfrac{\sqrt{\delta ^2-\epsilon}}{\delta+\sqrt{\epsilon}} \dfrac{(1-\sqrt{\epsilon}\,x^2)^2}{(1+\sqrt{\epsilon}\,x^2)^2},$
$ \varpi_3'=\dfrac{2y^2}{\sqrt{\beta_2}(1-\sqrt{\epsilon}x^2)^2}.$
}
     \end{tabular}
\end{table}
}
\end{tiny}
\subsection{$w$-function invariant on $<Q_2>$}
In this section, we consider a $w$-function that is invariant for the coset of the point $Q_2$, i.e. 
$w(P)=w(Q)$ for all points $Q$ in $<Q_2>+P$.
\begin{proposition}\label{P:J5}
		The Jacobi quartic curve $\J_{\epsilon,\delta }$ over the finite field $\F_q$ of odd characteristic with $\chi(\delta^2-\epsilon)=1$ has dADD-M formulas with the function $w$ given by $$w(x,y)=\dfrac{sx^2-\epsilon x^4}{1-sx^2},$$  where $s=-\delta-\sqrt{\delta^2-\epsilon}$ and 
		$ e={4(-\delta +\sqrt{\delta^2-\epsilon})}/{(-\delta -\sqrt{\delta^2-\epsilon})}$ .
		\end{proposition}
\begin{proof}
%
%
	From Proposition \eqref{JQTE} the Jacobi quartic curve $\J_{\epsilon,\delta }$ is isomorphic to twisted Edwards curve $\TE_{t,s}$ under the transformations
	
	\begin{equation*}
			\phi(x,y)=(x,\dfrac{y}{1-sx^2}).
			\end{equation*}
	From \cite{FH2}, the twisted Edwards curve $\TE_{t,s}$ has dADD-M formulas with the function $v$ by  $v(X,Y)=dX^2Y^2$ and $e=4t/s$.
So from Lemma \eqref{ISOw}, $\omega=v\circ \phi$ is a $w$-function. So we have 
$\omega(x,y)=v( \phi(x,y))=\frac{sx^2y^2}{(1-sx^2)^2}=\dfrac{sx^2-\epsilon x^4}{1-sx^2}$ and $ e={4(-\delta +\sqrt{\delta ^2-\epsilon })}/{(-\delta -\sqrt{\delta ^2-\epsilon})}$ . \qed  
	
%
		\end{proof}
Again, we have the same projective and mixed differential formulas as \eqref{eq:M2},\eqref{eq:M4} and for the Jacobi quartics curves $\J_{\epsilon,\delta}$ with
 $ \chi(2\sqrt{\delta ^2-\epsilon }~(-\delta+\sqrt{\delta ^2-\epsilon }))=1$, by assuming 
$r^2=\dfrac{2\sqrt{\delta ^2-\epsilon }}{-\delta+\sqrt{\delta ^2-\epsilon }}$,
we have the same mixed differential addition and doubling formulas as \eqref{eq:M6} with cost of $3\M+6\s+3\D$.
%
%
%
%
%

\begin{example}
Let $p=2^{221}-3$ and $w(x,y)=\frac{sx^2-\epsilon x^4}{1-sx^2}$. If $\delta=
4$, $\epsilon=29083947421295781$\\
$7461024832352724759012289132574602923421109182330$,  then  Jacobi quartic curve $\J_{\epsilon,\delta}$ over  $\F_p$ is of order $8\ell$, where $\ell$ is the prime 
\begin{gather*} 
\ell=42124916667422874679167211073468\\
\qquad 1763486441235699193747057509395729.
\end{gather*}
The cost of the mixed differential addition and doubling formulas \eqref{eq:M6} is
$3\M+6\s+3\D$, where $2\D$ is the multiplication by the small constant $r=2690$ and one $\D$ is the multiplication by $1/r$.
\end{example}

\begin{table}[h!]
	\caption{}\label{tab4}
	\renewcommand{\arraystretch}{2.25}
\begin{tabular}{l|l|l|l|l|l|l}
		\hline
	$w(P)$&$w(\cP_0)$ & $w(\cP_1)$ & $w(\cP_2)$ & $w(\cP_3)$ & $w(\cP_4)$ & \text{invariant on subgroup}\\
		\hline
	$\varpi_4=\frac{sx^2-\epsilon x^4}{1-sx^2} $&$\frac{1}{\varpi_4}$ &$\frac{1}{\varpi_4}$&$\varpi_4$ &$\varpi_4'$ &$\varpi_4''$ &$ <Q_2>$ $\cong$  $\Z_4$\\
	\hline
   \multicolumn{7}{c}{\tiny where
$ \varpi_4'=-\dfrac{\sqrt{\epsilon}(\sqrt{\beta_1} x+\sqrt{\alpha_1} y)^2+s(\beta_1 x^2-\alpha_1 y^2)}{\sqrt{\epsilon}(\sqrt{\beta_1} x-\sqrt{\alpha_1} y)^2+s(\beta_1 x^2-\alpha_1 y^2)},$
 $\varpi_4''=- \dfrac{\sqrt{\epsilon}(\sqrt{\beta_2} x+\sqrt{\alpha_2} y)^2-s(\beta_2 x^2-\alpha_2 y^2)}{\sqrt{\epsilon}(\sqrt{\beta_2} x-\sqrt{\alpha_2} y)^2-s(\beta_2 x^2-\alpha_2 y^2)}.$
}
 \end{tabular}
\end{table}
\subsection{$w$-function invariant on $<\mathcal{T},Q_3>$}
Now, we introduce a $w$-function that is invariant for the coset of the point $Q_3$, i.e. 
$w(P)=w(Q)$ for all points $Q$ in $<Q_3>+P$.
\begin{proposition}\label{P:J8}
The Jacobi quartic curve $\J_{\epsilon,\delta}$ over the finite field $\F_q$ of odd characteristic with  $\chi(\epsilon)=1$ has dADD-M formulas with the function $w$ given by $$w(x,y)=2(\delta -\sqrt{\epsilon}) \,\dfrac{x^2(1-\sqrt{\epsilon}\,x^2)^2}{y^2(1+\sqrt{\epsilon}\,x^2)^2}$$ and
 $e=\dfrac{4(\delta+\sqrt{\epsilon})}{\delta-\sqrt{\epsilon}}$.
\end{proposition}

\begin{proof}
From combination of isogenies (\ref{TEMon}) and (\ref{JM}), we can obtain the following isogeny between the Jacobi quartic curve $J_{\epsilon,\delta}$ and twisted Edwards curve $TE_{2(\delta+\sqrt{\epsilon}),2(\delta-\sqrt{\epsilon})}$ by the map
\begin{gather*}\label{JTE02}
\phi(x,y)=(\dfrac{-x}{y},\dfrac{1-\sqrt{\epsilon}x^2}{1+\sqrt{\epsilon} x^2}).
\end{gather*}

From \cite{FH2}, the twisted Edwards curve $TE_{2(\delta+\sqrt{\epsilon}),2(\delta-\sqrt{\epsilon})}$ has dADD-M formulas with the function $v$ by  $v(X,Y)=2(\delta-\sqrt{\epsilon})\,X^2Y^2$ and $e=4(\delta+\sqrt{\epsilon})/(\delta-\sqrt{\epsilon})$.
So from Lemma \eqref{ISOw}, $\omega=v\circ \phi$ is a $w$-function. So we have 
$\omega(x,y)=v( \phi(x,y))=2(\delta -\sqrt{\epsilon}) \,\dfrac{x^2(1-\sqrt{\epsilon}\,x^2)^2}{y^2(1+\sqrt{\epsilon}\,x^2)^2}$ and
 $e=\dfrac{4(\delta+\sqrt{\epsilon})}{\delta-\sqrt{\epsilon}}$ . \qed  
	\end {proof}

%
%
%

Again, we have the same mixed differential addition and doubling formulas as  \eqref{eq:M2} ,\eqref{eq:M4}. 
Also, for the Jacobi quartics curves $\J_{\epsilon,\delta}$ with$\chi(\epsilon)=\chi(\sqrt{\epsilon}(\delta+\sqrt{\epsilon}))=1$, by assuming
$r^2=\dfrac{4\sqrt{\epsilon}}{\delta+\sqrt{\epsilon}}$
, we have the same mixed projective differential addition and doubling formulas as \eqref{eq:M6} with cost of $3\M+6\s+3\D$.
Notice, $\delta=\frac{4-r^2}{r^2}\sqrt{\epsilon}$. So, if the parameter $r$ is chosen to be small then the cost of mixed differential formulas is $3\M+6\s+1\D$.
\begin{example}
Let $p=2^{255}-19$ and $\varpi_6=2(\delta -\sqrt{\epsilon}) \,\dfrac{x^2(1-\sqrt{\epsilon}\,x^2)^2}{y^2(1+\sqrt{\epsilon}\,x^2)^2}$. If $\epsilon=1$ and $\delta=
30944373591988444460751973135974026291107853553685711039716989364724196976793$, then the  Jacobi quartic curve $\J_{\epsilon,\delta}$ over  $\F_p$ is of order $4\ell$, where $\ell$ is the prime 
\begin{gather*}
\ell=14474011154664524427946373126085988\\
\qquad 481688154239458342155892795746573869440921
\end{gather*}
The cost of the mixed differential addition and doubling formulas \eqref{eq:M6} is
$3\M+6\s+3\D$, where $2\D$ is the multiplication by the small constant $r=2074$ and one $\D$ is the multiplication by $1/r$.
\end{example}
\begin{example}
Let $p=2^{255}-19$ and $\varpi_6=2(\delta -\sqrt{\epsilon}) \,\dfrac{x^2(1-\sqrt{\epsilon}\,x^2)^2}{y^2(1+\sqrt{\epsilon}\,x^2)^2}$. If $\epsilon=1$ and $\delta=
33574686343109322207842306118478095692616517608182583806876614746137736188301$, then the  Jacobi quartic curve $\J_{\epsilon,\delta}$ over  $\F_p$ is of order $8\ell$, where $\ell$ is the prime 
\begin{gather*}
\ell=7237005577332262213973186563042994\\
\qquad 240794681469389996117254795117961872363349 
\end{gather*}
The cost of the mixed differential addition and doubling formulas \eqref{eq:M6} is
$3\M+6\s+3\D$, where $2\D$ is the multiplication by the small constant $r=27593$ and one $\D$ is the multiplication by $1/r$.
\end{example}
\begin{table}[h!]
		\caption{}\label{tab5}
	\renewcommand{\arraystretch}{2.25}
\begin{equation*}
	\begin{array}{l|l|l|l|l|l|l}
		\hline
	w(P)~&~ w(\cP_0) & w(\cP_1) & w(\cP_2) & w(\cP_3) & w(\cP_4) & \text{invariant on subgroup}\\
		\hline
		\varpi_5=2(\delta -\sqrt{\epsilon}) \,\dfrac{x^2(1-\sqrt{\epsilon}\,x^2)^2}{y^2(1+\sqrt{\epsilon}\,x^2)^2} &\varpi_5 &\frac{1}{\varpi_5}&\frac{1}{\varpi_5}& \varpi_5 &\frac{1}{\varpi_5} & <Q_3,\mathcal{T}> \cong  \Z_2 \times \Z_4 \\
\hline
     \end{array}
		\end{equation*}
\end{table} 
\subsection{$w$-function invariant on $<Q_1,Q_2>$}
And finally, we introduce a $w$-function that is invariant for the coset of the points $Q_1$ and $Q_2$, i.e. 
$w(P)=w(Q)$ for all points $Q$ in $<Q_1,Q_2>+P$.
\begin{proposition}\label{P:J9}
The Jacobi quartic curve $\J_{\epsilon,\delta}$ over the finite field $\F_q$ of odd characteristic with  $\chi(\epsilon)=1$ has dADD-M formulas with the function $w$ given by $$w(x,y)=4 \sqrt{\epsilon} \,\dfrac{x^2y^2}{(1-\epsilon\,x^4)^2}$$ and
$e=\dfrac{2(\sqrt{\epsilon}-\delta)}{\sqrt{\epsilon}}$.
\end{proposition}

\begin{proof}

The  Jacobi quartic curve $J_{a^2,a-2d}$ and twisted Edwards curve $TE_{a,d}$ are isogenous \cite{AG,H15} by the map\\
\begin{gather*}\label{JTE0}
\varphi(x,y)=(\dfrac{2x}{1+ax^2},\dfrac{1-ax^2}{y}) .
\end{gather*}

From \cite{FH2}, twisted Edwards curve  $TE_{a,d}$ has dADD-M formulas with $v(X,Y)=a\, X^2/Y^2$ and $e=4d/a$. So from Lemma \eqref{ISOw},  Jacobi quartic curve $\J_{\epsilon,\delta}$ with $\chi(\epsilon)=1$, has dADD-M formulas  with 
 $$\omega(x,y)=v( \psi(x,y))=4 \sqrt{\epsilon} \,\dfrac{x^2y^2}{(1-\epsilon\,x^4)^2}$$ and
$e=\dfrac{2(\sqrt{\epsilon}-\delta)}{\sqrt{\epsilon}}$. \qed	
	\end {proof}
Similarly, we have the same mixed differential addition and doubling formulas as  \eqref{eq:M2} ,\eqref{eq:M4}. 
Also, for the Jacobi quartics curves $\J_{\epsilon,\delta}$ with $\chi(\epsilon)=\chi(\delta^2-\epsilon)=1$, by assuming
$r=\dfrac{\delta+\sqrt{\epsilon}}{\sqrt{\delta^2-\epsilon}}$
, we have the same mixed projective differential addition and doubling formulas as \eqref{eq:M6} with cost of $3\M+6\s+3\D$.
Notice, $\delta=\frac{r^2+1}{r^2-1}\sqrt{\epsilon}$. So, if the parameter $r$ is chosen to be small then the cost of mixed differential formulas is $3\M+6\s+1\D$. 
\begin{example}
Let $p=2^{221}-3$ and $\varpi_7(x,y)=4\sqrt{\epsilon}\, \frac{x^2y^2}{(1-\epsilon\,x^4)^2}$. If $\epsilon=1$ and $\delta= 1625485888214265059385149248953355680048842182580119166027190412745$, then $\chi(\delta^2-\epsilon)=1$  and the  Jacobi quartic curve $\J_{\epsilon,\delta}$ over  $\F_p$ is of order $4\ell$, where $\ell$ is the prime 
\begin{gather*} 
\ell=842498333348457493583344221469363\\
\qquad 786243439662766001461011538653297.
\end{gather*}
The cost of the mixed differential addition and doubling formulas \eqref{eq:M6} is
$3\M+6\s+3\D$, where $2\D$ is the multiplication by the small constant $r=3096$ and one $\D$ is the multiplication by $1/r$.
\end{example}

\begin{table}[h!]
		\caption{}\label{tab6}
	\renewcommand{\arraystretch}{2.25}
\begin{equation*}
	\begin{array}{l|l|l|l|l|l|l}
		\hline
	w(P)~&~ w(\cP_0) & w(\cP_1) & w(\cP_2) & w(\cP_3) & w(\cP_4)& \text{invariant on subgroup}\\
		\hline
	\varpi_6=4 \sqrt{\epsilon} \,\dfrac{x^2y^2}{(1-\epsilon\,x^4)^2} &\varpi_6 &\varpi_6 &\varpi_6 & \frac{1}{\varpi_6} &\frac{1}{\varpi_6} & <Q_1, Q_2> \cong \Z_2 \times \Z_4\\
	\hline
     \end{array}
		\end{equation*}
\end{table}
\subsection{$w$-function invariant on $\J[4](\F_q)$}
And finally, we introduce a $w$-function that is invariant for the coset of the full $4$-torsion group, i.e.,
$w(P)=w(Q)$ for all points $Q$ in $\J[4](\F_q)+P$.
\begin{proposition}\label{P:J10}
The Jacobi quartic curve $\J_{\epsilon,\delta}$ over the finite field $\F_q$ of odd characteristic with  $\chi(2\sqrt{\epsilon}(\sqrt{\epsilon}-\delta))=1$ has dADD-M formulas with the function $w$ given by $$w(x,y)=8\sqrt{2\sqrt{\epsilon}(\sqrt{\epsilon}-\delta)}\;\dfrac{x^2y^2(1-\epsilon x^4)^2}{\left( 4\sqrt{\epsilon}\, x^2y^2+(1-\epsilon x^4)^2 \right)^2}$$, 
and
 $e=(2\sqrt{\epsilon}+\sqrt{2\sqrt{\epsilon}(\sqrt{\epsilon}-\delta)}~)^2/(2\sqrt{\epsilon}\sqrt{2\sqrt{\epsilon}(\sqrt{\epsilon}-\delta)})$.
\end{proposition}
\begin{proof}
We know that doubling formula \eqref{eq:DbEJ} is an isogeny for Jacobi quartic curve $\J_{\epsilon,\delta}$   to itself, as follow:

$$\phi(x,y)= \left(\dfrac{2xy}  {1-\epsilon x^4},\,\dfrac{(y^2+2\delta x^2)(1+\epsilon x^4)+4\epsilon x^4}  {(1-\epsilon x^4)^2} \right).$$

But from Proposition \eqref{P:J7} the Jacobi quartic curve $\J_{\epsilon,\delta}$ over $\F_q$ has dADD-M formulas with the function $w$ given by 
$v(X,Y)=2\sqrt{2\sqrt{\epsilon}(\sqrt{\epsilon}-\delta)}\;(\dfrac{X}{1+\sqrt{\epsilon} X^2})^2$ and
 $e=(2\sqrt{\epsilon}+\sqrt{2\sqrt{\epsilon}(\sqrt{\epsilon}-\delta)}~)^2/(2\sqrt{\epsilon}\sqrt{2\sqrt{\epsilon}(\sqrt{\epsilon}-\delta)})$.
 So from Lemma \eqref{ISOw}, $\omega=v\circ \phi$ is a $w$-function. So 
$\omega(x,y)=v( \phi(x,y))=8\sqrt{2\sqrt{\epsilon}(\sqrt{\epsilon}-\delta)}\;\dfrac{x^2y^2(1-\epsilon x^4)^2}{\left( 4\sqrt{\epsilon}\, x^2y^2+(1-\epsilon x^4)^2 \right)^2}. $ \qed
\end{proof}

Similarly, we have the same mixed differential addition and doubling formulas as  \eqref{eq:M2} ,\eqref{eq:M4}. 
Also, for the Jacobi quartics curves $\J_{a^2,\delta}$ with $\chi(2\sqrt{\epsilon}(\sqrt{\epsilon}-\delta))=1$, by assuming
$r=\dfrac{2\sqrt{\epsilon}-\sqrt{2\sqrt{\epsilon}(\sqrt{\epsilon}-\delta)}}{2\sqrt{\epsilon}+\sqrt{2\sqrt{\epsilon}(\sqrt{\epsilon}-\delta)}}$
, we have the same mixed projective differential addition and doubling formulas as \eqref{eq:M6} with cost of $3\M+6\s+3\D$. So, if the parameter $r$ is chosen to be small then the cost of mixed differential formulas is $3\M+6\s+1\D$. 
\begin{example}
Let $p=2^{221}-3$ and $w(x,y)=8\sqrt{2\sqrt{\epsilon}(\sqrt{\epsilon}-\delta)}\;\dfrac{x^2y^2(1-\epsilon x^4)^2}{\left( 4\sqrt{\epsilon}\, x^2y^2+(1-\epsilon x^4)^2 \right)^2}$. If $\epsilon=1$ and $\delta=
4055600956685933415374204120978894115631100479487678178603$ $8032234854669273739$, then the  Jacobi quartic curve $\J_{\epsilon,\delta}$ over  $\F_p$ is of order $8\ell$, where $\ell$ is the prime 
\begin{gather*} 
\ell=723700557733226221397318656304299\\
\qquad 4240823162899814764622947667093616846653001.
\end{gather*}
The cost of the mixed differential addition and doubling formulas \eqref{eq:M6} is
$3\M+6\s+3\D$, where $2\D$ is the multiplication by the small constant $r=14295$ and one $\D$ is the multiplication by $1/r$.
\end{example}
\begin{scriptsize}
\begin{table}[h!]
		\caption{}\label{tab7}
	\renewcommand{\arraystretch}{2.25}
\begin{equation*}
	\begin{array}{l|l|l|l|l|l|l}
		\hline
	w(P)~&~ w(\cP_0) & w(\cP_1) & w(\cP_2) & w(\cP_3) & w(\cP_4 )& \text{invariant on subgroup}\\
		\hline
	\varpi_7=\dfrac{8\sqrt{2\sqrt{\epsilon}(\sqrt{\epsilon}-\delta)}\;x^2y^2(1-\epsilon x^4)^2}{\left( 4\sqrt{\epsilon}\, x^2y^2+(1-\epsilon x^4)^2 \right)^2} &\varpi_7 & \varpi_7 &\varpi_7 &\varpi_7 &\varpi_7 & \J[4] \cong  \Z_4 \times \Z_4\\
	\hline
     \end{array}
		\end{equation*}
	\end{table}
\end{scriptsize}
\begin{theorem}
	Let $E$ be an elliptic curve over $\F_q$ of odd characteristic. If $E(\F_q)$ has a subgroup of order 4, then $E$ has a dADD-M formulas. 
\end{theorem}
\begin{proof}
Let the elliptic curve $E$ over $\F_q$ has a subgroup of order 4, so it has either full 2-torsion $\F_q$-rational points or a $\F_q$-rational point of order 4.
We recall \cite{BJ} that any elliptic curve with a point of order 2 is $\F_q$-isomorphic to a Jacobi quartic curve $\J_{\epsilon,\delta}$. Then, there exist a $\F_q$-isomorphism 
$\phi :  E(\F_q) \mapsto \J_{\epsilon,\delta}(\F_q)$. If $E(\F_q)$ has full 2-torsion points then $\J_{\epsilon,\delta}(\F_q)$ does too. So, we have $\chi(\epsilon)=1$ (see \cite{BJ}). 
Also, if $E(\F_q)$ has a point of order 4 then $\J_{\epsilon,\delta}(\F_q)$ does the same. So, we have $\chi(\delta^2-\epsilon)=1$. 
From Propositions \ref{P:J2} and \ref{P:J4}, $\J_{\epsilon,\delta}$ has dADD-M formulas by some function $w$ on $\J_{\epsilon,\delta}$ if $\chi(\epsilon)=1$ and $\chi(\delta^2-\epsilon)=1$, respectively. Thus, $E$ has a dADD-M formulas by the function $w \circ \phi$ which complete the proof. \qed
\end{proof}
\section{Comparison with previous works } \label{sec-Compar}
Gaudry and Lubicz \cite{GL9} give a very efficient mixed differential addition Montgomery like formulas for Kummer line with the cost of $4\M+6\s+3\D$, and $3\M+6\s+3\D$ if the base point is affine. The Kummer line is linked to the Legendre curve $\E_\lambda: Y^2=X(X-1)(X-\lambda)$, where $\lambda=a^4/(a^4-b^4)$ and $(a:b)$ defines the Kummer line. The group order of the corresponding curve $\E_\lambda$ over $\F_q$ is divisible by 4, and in particular it has full 2-torsion subgroup. Castryck, Galbraith and Farashahi \cite{CGF} give the $y$-coordinate differential addition Montgomery-like formulas for Edwards curves with the cost of $6\M+4\s+1\D$, and $5\M+4\s+1\D$ if the base point is affine.
Gu et.al. give the $x^2$-coordinate differential addition Montgomery-like formulas for Jacobi quartic $\J_{1,a}$ with the cost of $5\M+4\s+1\D$ if the base point is affine. Notice, the family of Jacobi curves $\J_{1,a}$ is the same (up to $\F_q$-isomorphism) as the family of Legendre curves $\E_\lambda$ with $\lambda \in \F_q$.     
Farashahi and Hosseini  \cite{FH2} provides a new formulas for complete twisted Edwards curves $\E_{\TE,a,d}$ with the cost of $3\M+6\s+3\D$. They also provide formulas for twisted Edwards and Montgomery curve with full 2-torsion points with the cost of $3\M+6\s+3\D$.
Bernstein and Lange \cite{BERNLANG} provides a Kummer-line formulas for Montgomery curves with full 2-torsion subgroup with the cost of $3\M+6\s+3\D$. 

The family of Jacobi quartic curves properly includes (up to $\F_q$-isomorphism) all the families of Legendre, Edwards, twisted Edwards and Montgomery curves. Notice, the family of Montgomery and twisted Edwards curves do not cover all elliptic curves with full 2-torsion subgroup. Our proposed formulas for Jacobi curves are improved in terms of efficiency and speed. The Montgomery-like formulas are provided for complete Jacobi curves with the cost of $3\M+7\s+1\D$ and $3\M+6\s+3\D$ which gives further speed up if the parameters are chosen to be small. 

Several points should be considered for an efficient implementation of the point multiplication on a given curve over a suitable finite field. In fact, one of them is reducing the number of basic finite field operations. In this paper, the Montgomery-like formulas with the cost of $5\M+4\s+1\D$ are provided for all elliptic curves over finite fields with group order divisible by 4. 

In Table \ref{tab}, we compare our new differential addition formulas for family of Jacobi quartic  curves  with the known formulas for other forms of elliptic curves.

\begin{table}[h!]
	\caption{ \scriptsize Cost of differential addition and doubling for families
		of elliptic curves in odd characteristic}\label{tab}
	\centering
	\begin{tabular}{|l|l|l|}
		\hline
		 Model~~ & Projective differential~~ & Mixed differential~~ \\
		\hline\hline
		Montgomery~~ \cite{Mo}~~ &
		$6\M+4\s+1\D$~~&$5\M+4\s+1\D$~~\\
	Montgomery~~ \cite{FH2}~~ &
			$4\M+7\s+1\D$~~&$3\M+7\s+1\D$~~\\
	Montgomery~~ \cite{FH2,BERNLANG}~~ &
			$4\M+6\s+3\D$~~&$3\M+6\s+3\D$~~\\
		\hline
		Kummer curve~~\cite{GL9}~~&$4\M+6\s+3\D$~~&$3\M+6\s+3\D$ \\
		\hline
		Twisted Edwards curve &&\\
		\cite{FH2}~~ & $6\M+5\s+1\D$~~&$5\M+4\s+1\D$~~\\
	\cite{FH2}~~ & $4\M+7\s+1\D$~~&$3\M+7\s+1\D$~~\\
		\cite{FH2}~~ & $4\M+6\s+3\D$~~&$3\M+6\s+3\D$~~\\
		\hline
		Jacobi quartic curve&&\\
		$\epsilon=a^2$ \cite{GGX12}& $6\M+4\s+5\D$~~&$5\M+4\s+5\D$~~\\
		$\epsilon=1$ \cite{GGX12}& $6\M+4\s+1\D$~~&$5\M+4\s+1\D$~~\\
   Proposition: \eqref{P:J1},\eqref{P:J3}&$6\M+6\s+2\D$~~&$5\M+6\s+2\D$~\\
    	Proposition: \eqref{P:J2}, \eqref{P:J4}, \eqref{P:J7}, \eqref{P:J5},\eqref{P:J8},\eqref{P:J9},\eqref{P:J10}&$6\M+4\s+1\D$~~&$5\M+4\s+1\D$~\\
	    &$4\M+7\s+1\D$~~&$3\M+7\s+1\D$\\

	    &$4\M+6\s+3\D$~~&$3\M+6\s+3\D$\\
		\hline
	\end{tabular}
\end{table}

\end{document}